\documentclass[11pt,reqno]{amsart}
\usepackage{bbm}
\usepackage{}
\usepackage{amsfonts}
\usepackage{amssymb,bm}
\usepackage{color}
\usepackage{amsmath, amssymb, amsfonts, amsthm}
\DeclareMathOperator{\curl}{curl}

\DeclareMathOperator{\divg}{div}

\hoffset=- 2cm \voffset=- 0cm 
 \theoremstyle{plain}
\newtheorem{Thm}{Theorem}

\newtheorem{Prop}[Thm]{Proposition}
\newtheorem{Rem}[Thm]{Remark}
\newtheorem{Lem}[Thm]{Lemma}

\newtheorem{Def}[Thm]{Definition}

\usepackage{amssymb}
\usepackage[active]{srcltx}
\usepackage{color}

\newcommand {\p}{\partial}
\newcommand{\q}{\quad}
\newcommand{\qq}{\qquad}

\newcommand{\eq}{\begin{equation}}
\newcommand{\eeq}{\end{equation}}
\def\a{\alpha}

\def\curl{\text{\rm curl\,}}
\def\div{\text{\rm div\,}}

\def\lam{\lambda}

\def\O{\Omega}

\def\p{\partial}

\def\q{\quad}
\def\qq{\qquad}

\def\s{\sigma}

\def\v{\vskip}

\def\loc{\text{\rm loc}}

\def\B{\bold B}

\def\E{\bold E}
\def\e{\bold e}

\def\F{\bold F}

\def\mH{\mathcal H}
\def\H{\bold H}
\def\h{\bold h}

\def\u{\bold u}

\def\w{\bold w}

\def\z{\bold z}
\def\0{\bold 0}

\errorcontextlines=0 \numberwithin{equation}{section}
\numberwithin{Thm}{section}

\begin{document}

\large

\title[Thermoelectric Model]
{Existence and Regularity of Weak Solutions for a Thermoelectric Model}

\author{Xing-Bin Pan and Zhibing Zhang}

\address{Xing-Bin Pan: School of Mathematical Sciences, East China Normal University, and NYU-ECNU Institute
of Mathematical Sciences at NYU Shanghai,  Shanghai 200062, People's Republic of China.}

\email{xbpan@math.ecnu.edu.cn}

\address{Zhibing Zhang: School of Mathematics and Physics, Anhui University of Technology, Ma'anshan 243032, People's Republic of China.}

\email{zhibingzhang29@126.com}



\begin{abstract}
This paper concerns a time-independent thermoelectric model with two different boundary conditions. The model is a nonlinear coupled system of the Maxwell equations and an elliptic equation. By analyzing carefully the nonlinear structure of the equations, and with the help of the De Giorgi-Nash estimate for elliptic equations, we obtain existence of weak solutions on Lipschitz domains for general boundary data. Using Campanato's method, we establish regularity results of the weak solutions.
\end{abstract}

\maketitle


\section{Introduction}

\subsection{The system}\

This paper is devoted to study of existence, regularity and uniqueness of weak solutions of the following
\begin{equation}\label{PZ0}
\left\{\aligned
&\nabla\times[\sigma(u)^{-1}\nabla\times \H]=\0,\q \nabla\cdot \H=0\q & \text{in } \O,\\
-&\Delta u=\sigma(u)^{-1}|\nabla\times \H|^2\q &\text{in } \O,\\
&u=u^0,\q \nu\cdot\H=0,\q \nu\times[\sigma(u)^{-1}\nabla\times \H]=\nu\times\E^0\q &\text{on } \p\O,
\endaligned\right.
\end{equation}
and
\begin{equation}\label{Pan-eq}
\left\{\aligned
&\nabla\times[\sigma(u)^{-1}\nabla\times \H]=\0,\q \nabla\cdot \H=0\q &\text{in } \O,\\
-&\Delta u=\sigma(u)^{-1}|\nabla\times \H|^2\q &\text{in } \O,\\
&u=u^0,\q  \nu\times\H=\nu\times\H^0\q &\text{on } \p\O.
\endaligned\right.
\end{equation}
Here $\O$ is a bounded domain in $\Bbb R^3$ with a Lipschitz boundary $\p\O$,  $u$ is a scalar function and $\H$ is a vector field, $\nu$ is the unit outer normal vector on $\p\O$,
$\sigma$ is a continuous scalar function bounded both from above and away from zero,  $\E^0$ and $\H^0$ are given vector fields. In \eqref{PZ0} it is natural to  assume that $\nabla\times\E^0=\0$.  Let us emphasize that the boundary condition for $\H$ in \eqref{PZ0} is the actual electric boundary condition (namely the boundary condition of prescribing the tangential component of the electric field $\E$, and it follows from the Maxwell's equations and the Ohm's law  $\E=\s(u)^{-1}\nabla\times\H$), and the boundary condition for $\H$ in \eqref{Pan-eq} is prescribing the tangential boundary condition for the magnetic field.

Systems \eqref{PZ0} and \eqref{Pan-eq} are the time-independent version of the thermoelectric model derived in \cite{Yin1994}, which describes electromagnetism in a medium with the electrical conductivity $\sigma$ depending on the temperature $u$, i.e., $\sigma=\sigma(u)$. Assuming that the electric current $\bold J$ and the electric field $\E$ obey Ohm's law
$\mathbf{J}=\sigma(u)\mathbf{E}$, and taking the Joule heating
$$\mathbf{J}\cdot\mathbf{E}=\sigma(u)|\mathbf{E}|^2
$$
as heat source, Yin derived the equation for the temperature $u$ as follows:
$$
{\p u\over\p t}-\Delta u=\sigma(u)|\bold E|^2.
$$
Yin combined this equation with the Maxwell's equations
\begin{equation*}
\left\{\aligned
&\nabla\cdot\mathbf{B}=0,\\
&\nabla\times\mathbf{H}=\frac{4\pi\sigma(u)}{c}\mathbf{E},\\
&\nabla\times\mathbf{E}=-\frac{1}{c}\frac{\p\mathbf{B}}{\p t},
\endaligned\right.
\end{equation*}
where $\mathbf{B}$ represents the magnetic induction, $\mathbf{H}$ represents the magnetic field, and $c$ is the speed of light.
Assuming that the magnetic induction $\bold B$ equals $\mu \mathbf{H}$, where the magnetic permeability $\mu$ is constant, and normalizing the constants in the equations, Yin derived the following model:
\begin{equation}\label{Time}
\left\{\aligned
&{\p \H\over\p t}+\nabla\times[\sigma(u)^{-1}\nabla\times \H]=\0,\\
&\nabla\cdot \H=0,\\
&{\p u\over\p t}-\Delta u=\sigma(u)^{-1}|\nabla\times \H|^2.
\endaligned\right.
\end{equation}
For more details of the derivation and analysis results see interesting papers \cite{Yin1994, Yin1997}.

In this paper we consider the steady state of \eqref{Time}:
\begin{equation*}
\left\{\aligned
&\nabla\times[\sigma(u)^{-1}\nabla\times \H]=\0,\\
&\nabla\cdot \H=0,\\
-&\Delta u=\sigma(u)^{-1}|\nabla\times \H|^2,
\endaligned\right.
\end{equation*}
and we shall establish existence and regularity of the weak solutions under natural assumptions. We hope our mathematical results be helpful for the application of this model in physics and engineering and for computations.

If the domain is simply-connected, then there exists a potential function $\varphi$ such that
$$\sigma(u)^{-1}\nabla\times \H=\nabla\varphi,
$$
and the above system is reduced to
\begin{equation*}
\left\{\aligned
&\nabla\cdot[\sigma(u)\nabla\varphi]=0,\\
-&\Delta u=\sigma(u)|\nabla\varphi|^2.
\endaligned\right.
\end{equation*}
This simplified model was used to analyze the Joule heating of electrically conducting media,  see \cite{Young1986,Young1987,YAB1983} and the references therein. See also \cite{FFH1992,Cimatti1991} and the references therein for the use of this model in the thermistor problem with a current limiting device. However if the domain is multiply-connected, then such potential function does not exist, and such reduction is not possible.

Recently,  under the condition of small boundary data, Pan \cite{Pan2015} obtained existence of classical solutions of \eqref{Pan-eq}:
\begin{enumerate}
\item[(i)]If $\O$ is a simply connected domain, then \eqref{Pan-eq} has classical solutions if $u^0$ and $\nu\cdot\nabla\times \H^0_T$
are small (\cite[Theorem 4.8]{Pan2015}).
\item[(ii)]If $\O$ is a multiply connected domain, then \eqref{Pan-eq} has classical solutions if $u^0$ and $\H^0_T$ are small and $\nu\cdot\nabla\times \H^0_T=0$ (\cite[Theorem 4.9]{Pan2015}).
\end{enumerate}
The main purpose of this paper is to prove existence of weak solutions of \eqref{PZ0} and \eqref{Pan-eq} for general domains. We shall obtain existence results on Lipschitz domains and without the extra condition of small boundary data. We shall also study regularity of weak solutions of \eqref{PZ0} and \eqref{Pan-eq}. Since in the present case $u$ is the temperature, its boundedness is essential for the models to be physically meaningful. Fortunately, this is a simple corollary of regularity results.

Systems \eqref{PZ0} and \eqref{Pan-eq} are interesting to us also for their special type of nonlinear structure. Since the only difference between systems \eqref{PZ0} and \eqref{Pan-eq} is the boundary condition for $\H$, we illustrate this point on system \eqref{PZ0}. Due to the quadratic nonlinearity in $\nabla\times\H$ in the second equation of \eqref{PZ0}, the problem of regularity of the weak solutions is non-trivial. In fact, if $(u,\H)$ is a weak solution of \eqref{PZ0}, then $u$ can be viewed as a weak solution of the Laplace equation with the right-hand term $\sigma(u)^{-1}|\nabla\times\H|^2\in L^1(\O)$:
\begin{equation}\label{eq-u}
-\Delta u=\sigma(u)^{-1}|\nabla\times \H|^2\q \text{in } \O,\q
u=u^0\q \text{on } \p\O.
\end{equation}
It is well-known that regularity of the Laplace equation with an $L^1$ right-hand term is a complicated problem. In order to derive higher regularity of a weak solution $(u,\H)$ of \eqref{PZ0}, we shall first improve the regularity of $\H$. By the assumption $\nabla\times\E^0=\0$ and using the first equation of \eqref{PZ0}, we can write
\eq\label{decom-3.1}
\sigma(u)^{-1}\nabla\times\H-\E^0=\nabla\varphi+\h,
\eeq
for some $\varphi\in H_0^1(\O)$ and $\h\in \Bbb H_D(\O)$, where $\Bbb H_D(\O)$ is the space of harmonic Dirichlet fields (see section 2). Noting that $\nabla\cdot(\nabla\times\H)=0$, we derive that $\varphi$ is a weak solution of a linear problem with measurable coefficient $\sigma(u)$:
\begin{equation}\label{eqphi-1}
\left\{\aligned
&\nabla\cdot[\sigma(u)(\nabla\varphi+\h+\E^0)]=0\q &\text{in } \O,\\
&\varphi=0\q &\text{on } \p\O.
\endaligned\right.
\end{equation}
Since $\E^0\in L^q(\O,\Bbb R^3)$ for some $q>3$, we can use Campanato's method to get $\nabla\varphi\in L^{2,\mu}(\O)$ for some $\mu>1$, from which we derive $\nabla\times\H\in L^{2,\mu}(\O)$ and $\varphi\in C^{0,\alpha}(\overline{\O})$ for some $\alpha\in (0,1)$. Here $L^{2,\mu}$ denotes the Campanato space. Then we write the right-hand term in \eqref{eq-u} in the form (see \eqref{equ-6})
$$
\sigma(u)^{-1}|\nabla\times\H|^2=\nabla\cdot[(\varphi+\varphi^0)\nabla\times\H]+(\h+\h_1)\cdot\nabla\times\H,
$$
where $\varphi^0\in H^1(\O)$ with $\int_\O \varphi^0 dx=0$, $\h_1\in \Bbb H_N(\O)$ (the space of harmonic Neumann fields), and both $(\varphi+\varphi^0)\nabla\times\H$ and $(\h+\h_1)\cdot\nabla\times\H$ belong to $\in L^{2,\mu}(\O)$. So we can apply Lemma \ref{M-Campanato3} to improve the regularity of $u$.

 Let us mention that we will also re-write the equation for $u$ in various forms for other purposes, see for instance \eqref{equ-7} and \eqref{equ-5}.

Regarding existence results of \eqref{Pan-eq}, we mention that a corresponding problem with the boundary condition for $\H$ replaced by a full Dirichlet boundary condition $\H=\H^0$ on $\p\O$ has been studied by several authors.
Yin \cite{Yin1997} studied the steady states of \eqref{Time} under the full Dirichlet boundary condition for $\H$ but without the divergence-free condition on $\H$:
\begin{equation}\label{Yin-eq}
\left\{\aligned
&\nabla\times[\sigma(u)^{-1}\nabla\times \H]=\0\q & \text{in } \O,\\
-&\Delta u=\sigma(u)^{-1}|\nabla\times \H|^2\q &\text{in } \O,\\
&u=u^0,\q \H=\H^0 \q&\text{on } \p\O.
\endaligned\right.
\end{equation}
Among other results, Yin \cite[Theorem 5.4]{Yin1997} claimed existence of a weak solution $(u,\H)\in W^{1,q}(\O)\times H^1(\O,\Bbb R^3)$ to the system \eqref{Yin-eq} with $1<q<3/2$. See also \cite{Yin2004} for the study of a more general system.
Under both the full Dirichlet boundary condition and the divergence-free condition on $\H$, Kang and Kim \cite{KK2002,KK2011} proved that the weak solutions of \eqref{Yin-eq} are globally H\"{o}lder continuous.
 Hong, Tonegawa and Yassin \cite{HTY2008} studied this system in the setting of differential forms in higher dimensions and obtained partial regularity of the weak solutions.

For the magnetic field $\H$ in Maxwell equations, it is more natural to consider the type of prescribing the normal or tangential boundary condition, rather than prescribing the full value of $\H$ (see for instance \cite{DL1990, Ce, Pi}).
Moreover, due to the different boundary conditions on $\H$,  existence results for problem \eqref{Pan-eq} and problem \eqref{Yin-eq} are quite different,  see  \cite[Subsection IV.G]{Pan2015}.

\subsection{Main results}\

Assume the function $\sigma(s)$ satisfies the following condition:
\begin{equation}\label{cond-rho}
\sigma \;\;\text{is continuous on $\Bbb R$},\q \sigma_1\leq\sigma(s)\leq\sigma_2\q \text{for all $s\in\Bbb R$},
\end{equation}
where $\sigma_1\leq \sigma_2$ are two positive constants. {The notation of spaces used in the following theorems will be given in section 2.}

\begin{Thm}\label{main1}
Let $\O$ be a bounded Lipschitz domain in $\Bbb R^3$. Assume the function $\sigma$ satisfies  \eqref{cond-rho}, $u^0\in H^1(\O)$, and $\E^0\in L^q(\O, \Bbb R^3)$ for some $q>3$ with $\nabla\times\E^0=\0$ in $\O$.
Then \eqref{PZ0} has a weak solution $(u,\H)\in H^1(\O)\times [H_0(\divg0,\O)\cap H(\curl,\O)]$. Furthermore we have:
\begin{itemize}
\item[(i)] If $\O$ is of class $C^2$ and $u^0\in W^{1,q}(\O)$, then $(u,\H)\in W^{1,q}(\O)\times W^{1,q}(\O,\Bbb R^3)$;
\item[(ii)] If $\O$ is of class $C^{3,\a}$ and simply connected, $\sigma\in C^{1,\a}_{\loc}(\Bbb R)$ with $\a\in (0,1)$,  and if
\eq\label{u0E0alpha}
(u^0,\E^0)\in C^{2,\a}(\overline{\O})\times C^{1,\a}(\overline{\O},\Bbb R^3),
\eeq
then $(u, \H)\in C^{2,\a}(\overline{\O})\times C^{2,\a}(\overline{\O},\Bbb R^3).$
\item[(iii)] Assume that $\O$ is of class $C^2$, the function $\sigma$ is Lipschitz on $\Bbb R$, i.e., there exists a positive constant $L$ such that
\begin{equation}\label{Lip-L}
|\sigma(s)-\sigma(t)|\leq L|s-t|\text{ for any $s$, $t\in \Bbb R$, }
\end{equation}
and assume $u^0\in W^{1,q}(\O)$. Then there exists $\eta>0$ such that if
\eq\label{small-E0}
\|\E^0\|_{L^q(\O)}<\eta,
\eeq
then the weak solution of \eqref{PZ0} in the space $H^1(\O)\times [H_0(\divg0,\O)\cap H(\curl,\O)\cap\Bbb H_N(\O)^\perp]$ is unique.
\end{itemize}
\end{Thm}
Theorem \ref{main1} will be proved through Theorems \ref{PZ1}, \ref{Thm-Lp-reg}, \ref{Thm-Holder} and \ref{Thm-uniq}.

\begin{Thm}\label{main2}
Let $\O$ be a bounded Lipschitz domain in $\Bbb R^3$. Assume the function $\sigma$ satisfies  \eqref{cond-rho}, $u^0\in H^1(\O)$, and $\H^0\in W^{1,q}(\O,\Bbb R^3)$ for some $q>3$.
Then \eqref{Pan-eq} has a weak solution $(u,\H)\in H^1(\O)\times [H(\divg0,\O)\cap H(\curl,\O)]$. Furthermore we have:
\begin{itemize}
\item[(i)] If $\O$ is of class $C^2$ and $u^0\in W^{1,q}(\O)$, then $(u,\H)\in W^{1,q}(\O)\times W^{1,q}(\O,\Bbb R^3)$;
\item[(ii)] If $\O$ is of class $C^{3,\a}$ and without holes, $\sigma\in C^{1,\a}_{\loc}(\Bbb R)$ with $\a\in (0,1)$,  and if
\eq\label{u0H0alpha}
(u^0,\H^0)\in C^{2,\a}(\overline{\O})\times C^{2,\a}(\overline{\O},\Bbb R^3),
\eeq
then $(u, \H)\in C^{2,\a}(\overline{\O})\times C^{2,\a}(\overline{\O},\Bbb R^3).$
\item[(iii)]  Assume that $\O$ is of class $C^2$,  $\s$ satisfies \eqref{Lip-L}, and $u^0\in W^{1,q}(\O)$. Then there exists $\eta>0$ such that if
\eq\label{small-H0}
\|\nabla\times\H^0\|_{L^q(\O)}<\eta,
\eeq
then the weak solution of \eqref{Pan-eq} in the space $H^1(\O)\times [H(\divg0,\O)\cap H(\curl,\O)\cap(\H^0+\Bbb H_D(\O)^\perp)]$ is unique.
\end{itemize}
\end{Thm}

Theorem \ref{main2} will be proved in Section 6.

This paper is organized as follows. In Section 2 we collect some preliminary results that will be needed in the later sections. In Section 3 we prove existence of weak solutions for system \eqref{PZ0} by using Schauder's fixed point theorem. Regularity of the weak solutions is given in Section 4. Uniqueness under the condition of small boundary data is proved in Section 5. In Section 6 we discuss system \eqref{Pan-eq}.

\vskip0.1in

\section{Preliminaries}


Let $\O$ be a bounded Lipschitz domain in $\Bbb R^3$. Let $\mathcal{D}(\O,\Bbb R^3)$ denote the space of $3$-dimensional vector-valued functions that are infinitely differentiable and have compact supports in $\O$, and $\mathcal{D}'(\O,\Bbb R^n)$ denote its dual space.
We use $L^p(\O)$, $W^{k,p}(\O)$ and  $C^{k,\a}(\overline\O)$ to denote the usual Lebesgue spaces, Sobolev spaces and H\"older spaces for scalar functions, and use $L^p(\O,\Bbb R^3)$, $W^{k,p}(\O,\Bbb R^3)$ and  $C^{k,\a}(\overline\O,\Bbb R^3)$  to denote the corresponding spaces of vector fields. However we use  the same notation to denote both the norm of scalar functions and that for vector fields in the corresponding spaces. For instance, we write $\|\phi\|_{L^p(\O)}$ for $\phi\in L^p(\O)$ and write $\|\u\|_{L^p(\O)}$ for $\u\in L^p(\O,\Bbb R^3)$.

In the study of problems \eqref{PZ0} and \eqref{Pan-eq}, the topology of the domain $\O$ plays important roles. The domain topology is well represented by the spaces of harmonic Neumann fields $\Bbb H_N(\O)$ and of harmonic Dirichlet fields $\Bbb H_D(\O)$: \footnote{One may also denote $\Bbb H_N(\O)$ by $\Bbb H_1(\O)$, and denote $\Bbb H_D(\O)$ by $\Bbb H_2(\O)$.}
$$
\aligned
&\Bbb H_N(\O)=\{\u\in L^2(\O,\Bbb R^3): \nabla\times\u=\0,\; \nabla\cdot\u=0\text{ in $\O$,}\; \nu\cdot\u=0 \text{ on $\p\O$}\},\\
&\Bbb H_D(\O)=\{\u\in L^2(\O,\Bbb R^3): \nabla\times\u=\0,\;\nabla\cdot\u=0\text{ in $\O$,}\; \nu\times\u=\0 \text{ on $\p\O$}\}.
\endaligned
$$
The dimension of $\Bbb H_N(\O)$ is equal to the number of `handles' of $\O$, and the dimension of $\Bbb H_D(\O)$ is equal to the number of holes in $\O$. In particular, if $\O$ is simply-connected, i.e. if $\O$ has no `handles', then $\Bbb H_N(\O)=\{\0\}$; and if $\O$ has no holes, then $\Bbb H_D(\O)=\{\0\}$.
The harmonic Neumann or Dirichlet fields enjoy good regularities. In fact, for $C^{1,1}$ domains we have $\Bbb H_N(\O)$, $\Bbb H_D(\O)\subset W^{1,p}(\O,\Bbb R^3)$ for any $1<p<\infty$, see \cite[Corollary 4.1, Corollary 4.2]{AS2013}. Thus by Morrey embedding, we have $\Bbb H_N(\O)$, $\Bbb H_D(\O)\subset C^{0,1-3/p}(\overline{\O},\Bbb R^3)$ for any $3<p<\infty$. Furthermore, if $\O$ is of class $C^{r,\theta}$, where $r\geq 2$ and $0<\theta<1$, then
$\Bbb H_N(\O)$, $\Bbb H_D(\O)\subset C^{r-1,\theta}(\overline{\O},\Bbb R^3)$, see \cite[p. 219-222]{DL1990}.

We also use the following notation:
$$\aligned
&H(\divg0,\O)=\{\u\in L^2(\O,\Bbb R^3): \nabla\cdot\u=0\;\text{in }\O\},\\
&H_0(\divg0,\O)=\{\u\in H(\divg0,\O):  \nu\cdot\u=0\;\text{on }\p\O\},\\
&H(\curl,\O)=\{\u\in L^2(\O,\Bbb R^3): \nabla\times\u\in L^2(\O,\Bbb R^3)\},\\
&H_0(\curl,\O)=\{\u\in H(\curl,\O): \nu\times\u=\0\;\text{on }\p\O\}.
\endaligned
$$
For $1<p<\infty$ we denote
$$
\mH^p(\O,\curl,\div)=\{\u\in L^p(\O,\Bbb R^3): \nabla\times\u\in L^p(\O,\Bbb R^3),\; \nabla\cdot\u\in L^p(\O)\}.
$$
If $X(\O)$ denotes a space of scalar functions defined on $\O$, then we write
$$\dot X(\O)=\{\phi\in X(\O): \int_\O \phi(x) dx=0\}.
$$

Recall the Helmholtz-Weyl decompositions of $L^2(\O,\Bbb R^3)$ on Lipschitz domain (see \cite[Theorem 5.3]{BPS2016} or \cite[(1.19) on p.156]{Pi1984}):
$$
\aligned
&L^2(\O,\Bbb R^3)=[\nabla \times H(\curl,\O)]\oplus\nabla H_0^1(\O)\oplus\Bbb H_D(\O),\\
&L^2(\O,\Bbb R^3)=[\nabla \times H_0(\curl,\O)]\oplus\nabla H^1(\O)\oplus\Bbb H_N(\O).
\endaligned
$$
As a direct corollary, we have the following decompositions for curl-free vector fields, which will be used frequently in this paper:

\begin{Lem}\label{curl0}
Assume $\O$ is a bounded Lipschitz domain in $\Bbb R^3$, and $\u\in L^2(\O,\Bbb R^3)$ with $\nabla\times\u=\0$ in $\O$.  Then the following conclusions are true.
\begin{enumerate}
\item[(i)] $\u$ can be decomposed in the form
$\u=\nabla\varphi+\h$, where $\varphi\in \dot H^1(\O)$ and $\h\in \Bbb H_N(\O)$.
\item[(ii)] If furthermore $\nu\times\u=\0$ on $\p\O$ in the sense of trace, then $\u$ can be decomposed in the form $\u=\nabla\varphi+\h$, where $\varphi\in H_0^1(\O)$ and $\h\in \Bbb H_D(\O)$.
\end{enumerate}
\end{Lem}

A regularity result for $\divg$-$\curl$ system is also an important ingredient in our proof of existence of weak solutions to \eqref{PZ0} and \eqref{Pan-eq}.
We use $H^{s,p}(\O,\Bbb R^3)$ to denote fractional order Sobolev space.

\begin{Lem}\label{Hsp} $($\cite[Theorem 11.2]{MMT2001}$)$
For any bounded Lipschitz domain $\O$ in $\Bbb R^3$, there exists $\epsilon=\epsilon(\O)>0$ with the following significance. Let $p\in (2-\epsilon,2+\epsilon)$. Set $s=1/p$ if $p\geq2$ and $s=1-1/p$ if $p\leq2$. Assume that $\u\in \mH^p(\O,\curl,\div)$ is such that either $\nu\cdot\u\in L^p(\p\O)$ or $\nu\times\u\in L^p(\p\O,\Bbb R^3)$. Then
$$\u\in \bigcap_{\mu>0}H^{s-\mu,p}(\O,\Bbb R^3),
$$
and, for each $\mu>0$, there exists $C=C(\mu,p,\O)$ such that
$$
\aligned
\|\u\|_{H^{s-\mu,p}(\O)}&\leq C(\|\u\|_{L^p(\O)}+\|\nabla\cdot\u\|_{L^p(\O)}+\|\nabla\times\u\|_{L^p(\O)}\\
&\qq +\min\{\|\nu\cdot\u\|_{L^p(\p\O)},\|\nu\times\u\|_{L^p(\p\O)}\}).
\endaligned
$$
When $p=2$,  the above inequality remains true for $\mu=0$.
\footnote{The conclusion in the case where $p=2$ has been also obtained by Costabel \cite{Co1990}.}
\end{Lem}

We use $L^{2,\mu}(\O)$ to denote a Campanato space, which consists of scalar functions satisfying
$$\|u\|_{L^{2,\mu}(\O)}=\Big(\|u\|_{L^2(\O)}^2+\sup_{{x_0\in \overline{\O},}\atop {0<r<\infty}}r^{-\mu}\int_{\O_r(x_0)}|u-u_{x_0,r}|^2dx\Big)^{1/2}<\infty,
$$
where
$$\O_r(x_0)=\O\cap B_r(x_0),\q
u_{x_0,r}=\frac{1}{|\O_r(x_0)|}\int_{\O_r(x_0)}u(x)dx.
$$
Campanato spaces play a key role in our proof of regularity of weak solutions to \eqref{PZ0} and \eqref{Pan-eq}.
Below we list some properties for Campanato spaces, which can be found in \cite[Theorem 1.17, Lemma 1.19, Theorem 1.40]{Tro1987}.

\begin{Lem}\label{M-Campanato}
Assume $\O$ is a bounded $C^1$ domain in $\Bbb R^n$.
\begin{itemize}
\item[(i)] Let $0\leq \mu<n$. Then the mapping
$$u\mapsto \Big(\sup_{{x_0\in \overline{\O}}\atop {0<r<\infty}}r^{-\mu}\int_{\O_r(x_0)}u^2dx\Big)^{1/2}
$$
defines an equivalent norm on $L^{2,\mu}(\O)$. Hence $L^\infty(\O)$ is a space of multipliers for $L^{2,\mu}(\O)$. That is to say, for any $u\in L^{2,\mu}(\O)$ and any $v\in L^\infty(\O)$, we have
$$\|uv\|_{L^{2,\mu}(\O)}\leq C(n,\mu,\O)\|u\|_{L^{2,\mu}(\O)}\|v\|_{L^\infty(\O)}.
$$

\item[(ii)] Let $n<\mu\leq n+2$. Then $L^{2,\mu}(\O)$ is isomorphic to $C^{0,\delta}(\overline{\O})$ for $\delta=(\mu-n)/2$.

\item[(iii)]Let $0\leq \mu<n$. If $u\in H^1(\O)$ and $\nabla u\in L^{2,\mu}(\O)$, then $u\in L^{2,2+\mu}(\O)$ with
$$\|u\|_{L^{2,2+\mu}(\O)}\leq C(n,\mu,\O)(\|u\|_{L^2(\O)}+\|\nabla u\|_{L^{2,\mu}(\O)}).
$$

\item[(iv)] We have the following embedding:
$$\aligned
&L^{2,\lambda}(\O)\hookrightarrow L^{2,\mu}(\O)\q\;\text{\rm if } 0\leq \mu<\lambda\leq n+2,\\
&L^p(\O)\hookrightarrow L^{2,\mu}(\O)\q\q\text{\rm if }p>2,\;\; \mu=n(p-2)/p.
\endaligned
$$
\end{itemize}
\end{Lem}

The $L^{2,\mu}$ regularity of first derivatives for the Dirichlet problem
\begin{equation}\label{BD}
\nabla\cdot (A\nabla u)=f+\nabla\cdot \F\q \text{\rm in } \O,\q
 u=u^0\q\text{\rm on } \p\O,
\end{equation}
and for the Neumann problem
\begin{equation}\label{BN}
\nabla\cdot (A\nabla u)=\nabla\cdot \F\q\text{\rm in } \O,\q
\nu\cdot(A\nabla u)=\nu\cdot\F\q\text{\rm on } \p\O,
\end{equation}
can be derived by Campanato's method, see \cite[Theorem 2.19]{Tro1987}.

\begin{Lem}\label{M-Campanato3}
Let $n\geq 3$ and $\O$ be a bounded $C^1$ domain in $\Bbb R^n$. Suppose the matrix-valued function $A$ satisfies
$$\lambda|\xi|^2\leq \langle A\xi,\xi\rangle\leq \Lambda|\xi|^2,\q \forall \xi\in\Bbb R^n,
$$
where $0<\lam\leq\Lambda<\infty$. There exist constants $C>0$ and $\delta\in (0,1)$, both depending only on $\O,\lambda,\Lambda$, such that for any $0<\mu<n-2+2\delta$, if
$$f\in L^{2,(\mu-2)^+}(\O),\q \F\in L^{2,\mu}(\O),\q u^0\in H^1(\O),\q \nabla u^0\in L^{2,\mu}(\O),
$$
and if $u\in H^1(\O)$ is a weak solution of \eqref{BD}, then $\nabla u\in L^{2,\mu}(\O)$, and we have the estimate
$$\|\nabla u\|_{L^{2,\mu}(\O)}\leq C\{\|u\|_{H^1(\O)}+\|f\|_{L^{2,(\mu-2)^+}(\O)}+\|\F\|_{L^{2,\mu}(\O)}+\|\nabla u^0\|_{L^{2,\mu}(\O)}\}.
$$
\end{Lem}

\begin{Lem}\label{M-Campanato3.5}
Assume $\O$ and $A$ satisfy the conditions in Lemma \ref{M-Campanato3}. There exist constants $C>0$ and $\delta\in(0,1)$, both depending only on $\O,\lambda,\Lambda$, such that for $0<\mu<n-2+2\delta$,  if $\F\in L^{2,\mu}(\O)$, and if $u\in H^1(\O)$ is a weak solution of \eqref{BN},
then $\nabla u\in L^{2,\mu}(\O)$, and we have the estimate
$$\|\nabla u\|_{L^{2,\mu}(\O)}\leq C\{\|u\|_{H^1(\O)}+\|\F\|_{L^{2,\mu}(\O)}\}.
$$
\end{Lem}

\vskip0.1in

\section{Existence of Weak Solutions}\label{sec3}

\begin{Def}
We say that $(u,\H)\in H^1(\O)\times [H_0(\divg0,\O)\cap H(\curl,\O)]$ is a weak solution of \eqref{PZ0} if $u=u^0$ on $\p\O$ in the sense of trace, and if $$\aligned
&\int_\O\nabla u\cdot\nabla v\, dx=\int_\O\sigma(u)^{-1}|\nabla\times \H|^2v\, dx,\q&\forall v \in H^1_0(\O)\cap L^\infty(\O),\\
&\int_\O[\sigma(u)^{-1}\nabla\times \H]\cdot \nabla\times\w\, dx=\int_\O \mathbf{E}^0\cdot\nabla\times\w\, dx,\q&\forall \w\in H_0(\divg0,\O)\cap H(\curl,\O).
\endaligned
$$
\end{Def}

Proof of the existence result in Theorem \ref{PZ1} needs the following lemma, which will also be needed in the proof of Theorem \ref{Thm-Lp-reg} in the next section.

\begin{Lem}\label{lemma3.3}
Let $\O$ be a bounded Lipschitz domain in $\Bbb R^3$. Assume that the function $\sigma$ satisfies  \eqref{cond-rho}, $u^0\in H^1(\O)$, and $\E^0\in L^q(\O, \Bbb R^3)$ for some $q>3$  with $\nabla\times\E^0=\0$ in $\O$. For any given $w\in L^2(\O)$, the following system
\begin{equation*}
\left\{\aligned
&\nabla\times[\sigma(w)^{-1}\nabla\times \H]=\0,\q \nabla\cdot \H=0\q & \text{\rm in } \O,\\
-&\Delta u=\sigma(w)^{-1}|\nabla\times \H|^2\q &\text{\rm in } \O,\\
&u=u^0,\q \nu\cdot\H=0,\q \nu\times[\sigma(w)^{-1}\nabla\times \H]=\nu\times\E^0\q & \text{\rm on } \p\O,
\endaligned\right.
\end{equation*}
has a unique weak solution $(u_w,\H_w)\in H^1(\O)\times [H_0(\divg0,\O)\cap H(\curl,\O)\cap\Bbb H_N(\O)^\perp]$
with the estimates
\begin{equation}\label{M-C1}
\|\nabla\times\H_w\|_{L^2(\O)}\leq\sigma_2\|\E^0\|_{L^2(\O)},
\end{equation}
\begin{equation}\label{Const-K}
\aligned
\|u_w\|_{L^2(\O)}\leq C(\O,q,\sigma_1,\sigma_2)\left\{\|\E^0\|_{L^q(\O)}\|\E^0\|_{L^2(\O)}+\|u^0\|_{H^1(\O)}\right\}.
\endaligned
\end{equation}
\end{Lem}

\begin{proof}
{\it Step 1}. For any given $w\in L^2(\O)$, let $\H_w\in H_0(\divg0,\O)\cap H(\curl,\O)\cap\Bbb H_N(\O)^\perp$ be the unique weak solution of the following system
\begin{equation}\label{equation-H}
\left\{\aligned
&\nabla\times[\sigma(w)^{-1}\nabla\times \H_w]=\0,\q \nabla\cdot \H_w=0\q &\text{in } \O,\\
&\nu\cdot\H_w=0,\q \nu\times[\sigma(w)^{-1}\nabla\times \H_w]=\nu\times\E^0\q &\text{on } \p\O.
\endaligned\right.
\end{equation}
Existence of a unique weak solution $\H_w$ of \eqref{equation-H} can be proved by using the Lax-Milgram theorem, with the help of Poincar\'{e} type inequality
$$\|\bold v\|_{L^2(\O)}\leq C(\O)\|\nabla\times\bold v\|_{L^2(\O)} \q\forall \bold v\in H_0(\divg0,\O)\cap H(\curl,\O)\cap\Bbb H_N(\O)^\perp,$$
which is a consequence of a compact embedding theorem in Lipschitz domains established in \cite{Pi1984}, see also \cite{BPS2016} for weak Lipschitz domains and mixed boundary conditions.
Taking $\H_w$ as a test function  to \eqref{equation-H} and using condition \eqref{cond-rho}, we obtain the estimate \eqref{M-C1}.

\vskip0.05in

{\it Step 2}. For the weak solution $\H_w$ of \eqref{equation-H} obtained above, we have
$$\nabla\times[\sigma(w)^{-1}\nabla\times\H_w-\E^0]=\0\;\;\text{in }\O,\q
\nu\times[\sigma(w)^{-1}\nabla\times\H_w-\E^0]=\0\;\;\text{on }\p\O.
$$
It follows from Lemma \ref{curl0} that there exist $\varphi_w\in H_0^1(\O)$ and $\h_w\in \Bbb H_D(\O)$ such that
\eq\label{decomh}
\sigma(w)^{-1}\nabla\times\H_w-\E^0=\nabla\varphi_w+\h_w,
\eeq
and $\varphi_w$ satisfies the equation
\eq\label{eq-phi}
\left\{\aligned
&\Delta\varphi_w=\nabla\cdot [\sigma(w)^{-1}\nabla\times\H_w-\E^0]\q &\text{in } \O,\\
&\varphi_w=0\q &\text{on } \p\O.
\endaligned\right.
\eeq
We show that, there exists $C_1=C_1(\O,q,\sigma_1,\sigma_2)>0$ such that
\begin{equation}\label{M-C4}
\|\varphi_w\|_{L^\infty(\O)} \leq C_1\|\E^0\|_{L^q(\O)}.
\end{equation}

To prove \eqref{M-C4}, we note that the vector field $\h_w$ in \eqref{decomh} can be written as follows:
\eq\label{formula-cj}
\h_w=\sum_{j=1}^mc_j\e_j,\q c_j=\int_\O\e_j\cdot(\sigma(w)^{-1}\nabla\times\H_w-\E^0) dx,\q j=1,
\cdots, m,
\eeq
where $\{\e_1,\cdots,\e_m\}$ is an orthonormal basis of $\Bbb H_D(\O)$ with respect to the $L^2$-norm. \eqref{formula-cj} can be verified by using the $L^2$ orthogonality of $\nabla\varphi_w$ and $\h_w$. By Lemma \ref{Hsp}, we get $\h_w\in H^{\frac{1}{2}}(\O,\mathbb{R}^3)$. Hence the Sobolev embedding implies that $\h_w\in L^3(\O,\mathbb{R}^3)$.  Let $\epsilon$ be the constant in Lemma \ref{Hsp}. We choose $p\in (2,\min\{2+\epsilon,3\})$. By Lemma \ref{Hsp} again, we obtain
$$\h_w\in H^{\frac{1}{p}-\mu,p}(\O),\q\text{ where $\mu=\frac{p-2}{2p}.$}$$
By the Sobolev embedding, we get $\h_w\in L^{\frac{6p}{p+2}}(\O,\mathbb{R}^3)$ with the estimate
\begin{equation}\label{est-h}
\|\h_w\|_{L^{\frac{6p}{p+2}}(\O)}\leq C(p,\O)\|\h_w\|_{L^2(\O)}.
\end{equation}
On the other hand, for any $\zeta\in H_0^1(\O)$, we have
\begin{equation}\label{fai-heq}
\aligned
\int_\O\sigma(w)(\nabla\varphi_w+\h_w+\E^0)\cdot\nabla\zeta dx&=\int_\O (\nabla\times\H_w)\cdot\nabla\zeta dx
=0.
\endaligned
\end{equation}
Hence $\varphi_w$ is also a solution of the following equation
\begin{equation}\label{equation-fai}
\left\{\aligned
&\nabla\cdot[\sigma(w)(\nabla\varphi_w+\h_w+\E^0)]=0\q &\text{in } \O,\\
&\varphi_w=0\q &\text{on } \p\O.
\endaligned\right.
\end{equation}
Applying \cite[Theorem 8.16]{GT2001} to \eqref{equation-fai} we obtain
\begin{equation}\label{est-fai}
\|\varphi_w\|_{L^\infty(\O)}\leq C(\O,q,\sigma_1,\sigma_2)\|\h_w+\E^0\|_{L^{q_1}(\O)},\q\text{ where $q_1=\min\{q,\frac{6p}{p+2}\}>3$.}
\end{equation}
Then \eqref{M-C4} follows from \eqref{M-C1}, \eqref{formula-cj}, \eqref{est-h}, \eqref{est-fai}.

Furthermore, applying \cite[Theorem 8.29]{GT2001} instead of \cite[Theorem 8.16]{GT2001}, we see that there exist constants $\alpha=\alpha(\O,q,\sigma_1,\sigma_2)\in (0,1)$ and $C=C(\O,q,\sigma_1,\sigma_2)$ such that
\begin{equation}\label{fai-c0a}
\|\varphi_w\|_{C^{0,\alpha}(\overline{\O})} \leq C\|\E^0\|_{L^q(\O)}.
\end{equation}

\vskip0.05in

{\it Step 3}. For the $w$ and $\H_w$ given above, we show that the following equation has a unique $H^1$ weak solution $u_w$:
\begin{equation}\label{equation-u}
-\Delta u_w=\sigma(w)^{-1}|\nabla\times\H_w|^2\q \text{in } \O,\q
u_w=u^0 \q\text{on } \p\O.
\end{equation}
To prove this, note that from \eqref{decomh} we have
$$\sigma(w)(\nabla\varphi_w+\h_w+\E^0)\cdot \nabla\varphi_w=\nabla\cdot[\varphi_w\sigma(w)(\nabla\varphi_w+\h_w+\E^0)].
$$
This equality can also be verified by taking $\zeta=\varphi_w v$ in \eqref{fai-heq}, with an arbitrary $v\in\mathcal{D}(\O)$.
Hence
$$
\aligned
&\sigma(w)^{-1}|\nabla\times\H_w|^2=\sigma(w)|\nabla\varphi_w+\h_w+\E^0|^2\\
=&\nabla\cdot[\varphi_w\sigma(w)(\nabla\varphi_w+\h_w+\E^0)]+\sigma(w)(\nabla\varphi_w+\h_w+\E^0)\cdot(\h_w+\E^0)\\
=&\nabla\cdot(\varphi_w\,\nabla\times\H_w)+(\h_w+\E^0)\cdot\nabla\times\H_w,
\endaligned
$$
from which we see that $\sigma(w)^{-1}|\nabla\times\H_w|^2\in H^{-1}(\O)$. Therefore, by Lax-Milgram theorem, the Dirichlet problem
\eqref{equation-u} has a unique weak solution $u_w\in H^1(\O)$.

To prove \eqref{Const-K}, write \eqref{equation-u} in the following form
\begin{equation}\label{equ-7}
\left\{\aligned
-&\Delta (u_w-u^0)=\nabla\cdot(\varphi_w\,\nabla\times\H_w+\nabla u^0)+(\h_w+\E^0)\cdot\nabla\times\H_w\q& \text{in } \O,\\
& u_w-u^0=0 \q&\text{on } \p\O.
\endaligned\right.
\end{equation}
Taking $u_w-u^0$ as a test function, we get
$$
\aligned
\|\nabla (u_w-u^0)\|_{L^2(\O)}&\leq \|\varphi_w\,\nabla\times\H_w+\nabla u^0\|_{L^2(\O)}+C(\O)\|(\h_w+\E^0)\cdot\nabla\times\H_w\|_{L^\frac{6}{5}(\O)}\\
&\leq C(\O)\left(\|\varphi_w\|_{L^\infty(\O)}+\|\h_w+\E^0\|_{L^3(\O)}\right)\|\nabla\times\H_w\|_{L^2(\O)}+\|\nabla u^0\|_{L^2(\O)}.
\endaligned
$$
From this, \eqref{M-C1}, \eqref{M-C4}, and by Poincar\'{e} inequality, we find that
$$
\aligned
\|u_w\|_{L^2(\O)}&\leq \|u_w-u^0\|_{L^2(\O)}+\|u^0\|_{L^2(\O)}\\
&\leq C\left(\|\E^0\|_{L^q(\O)}\|\E^0\|_{L^2(\O)}+\|u^0\|_{H^1(\O)}\right),
\endaligned
$$
where the constant $C$ depends only on $\O, q, \sigma_1,\sigma_2$. This gives the estimate \eqref{Const-K}.

\end{proof}

\begin{Thm}\label{PZ1}
Let $\O$ be a bounded Lipschitz domain in $\Bbb R^3$. Assume that the function $\sigma$ satisfies  \eqref{cond-rho}, $u^0\in H^1(\O)$, and $\E^0\in L^q(\O, \Bbb R^3)$ for some $q>3$  with $\nabla\times\E^0=\0$ in $\O$.
Then \eqref{PZ0} has a weak solution $(u,\H)\in H^1(\O)\times [H_0(\divg0,\O)\cap H(\curl,\O)]$.
\end{Thm}

\begin{proof}

We define an operator $\mathrm{T}: L^2(\O)\to L^2(\O)$ as follows.
Given $w\in L^2(\O)$, we define $\H_w$, $\varphi_w$ and $u_w$ by the solution of \eqref{equation-H}, \eqref{eq-phi} and \eqref{equation-u} successively, and then define $\mathrm{T}(w)=u_w$. Recall that the estimates of $\|\varphi_w\|_{C^{0,\alpha}(\overline{\O})}$ and $\|u_w\|_{L^2(\O)}$ (see \eqref{fai-c0a} and \eqref{Const-K}) do not depend on the choice of $w\in L^2(\O)$, and these uniform (in $w$) estimates will be crucial in the proof of  existence of a fixed point of $T$.
Denote by $K$ the right hand side of the inequality \eqref{Const-K} and let
$$D=\{w\in L^2(\O):\|w\|_{L^2(\O)}\leq K\}.
$$
Obviously, $D$ is  convex and closed in $L^2(\O)$, and $\mathrm{T}$ maps $D$ into $D$.

We show that $\mathrm{T}$ is continuous from $L^2(\O)$ to $H^1(\O)$. Suppose $w_k\rightarrow w_0$ in $L^2(\O)$ as $k\rightarrow\infty$. Denote by $\H_k, \varphi_k, u_k$ and $\H_0, \varphi_0, u_0$ the solutions $\H_{w_k}, \varphi_{w_k}, u_{w_k}$ and $\H_{w_0}, \varphi_{w_0}, u_{w_0}$ obtained by setting $w=w_k$ and $w=w_0$ in the equation \eqref{equation-H}, \eqref{eq-phi} and \eqref{equation-u}, respectively. Then we obtain
\begin{equation}\label{eq-HH}
\left\{\aligned
&\nabla\times[\sigma(w_k)^{-1}\nabla\times (\H_k-\H_0)]=\nabla\times[(\sigma(w_0)^{-1}-\sigma(w_k)^{-1})\nabla\times\H_0]\q & \text{in } \O,\\
&\nabla\cdot (\H_k-\H_0)=0\q &\text{in } \O,\\
&\nu\cdot(\H_k-\H_0)=\0\q&\text{on } \p\O,\\
&\nu\times[\sigma(w_k)^{-1}\nabla\times (\H_k-\H_0)]=\nu\times[(\sigma(w_0)^{-1}-\sigma(w_k)^{-1})\nabla\times\H_0]\q&\text{on } \p\O.
\endaligned\right.
\end{equation}
By the Lebesgue's dominated convergence theorem we have
$$\|(\sigma(w_0)^{-1}-\sigma(w_k)^{-1})\nabla\times\H_0\|_{L^2(\O)}\rightarrow0\q\text{as }k\rightarrow\infty.
$$
Then by the $L^2$ estimate of \eqref{eq-HH} and using condition \eqref{cond-rho} we find that
\begin{equation}\label{converg1}
\aligned
\|\nabla\times (\H_k-\H_0)\|_{L^2(\O)}&\leq \sigma_2\|(\sigma(w_0)^{-1}-\sigma(w_k)^{-1})\nabla\times\H_0\|_{L^2(\O)}\rightarrow 0.
\endaligned
\end{equation}
Thus we have
\begin{equation}\label{converg2}
\aligned
&\|\sigma(w_k)^{-1}\nabla\times\H_k-\sigma(w_0)^{-1}\nabla\times\H_0\|_{L^2(\O)}\\
\leq& \|\sigma(w_k)^{-1}\nabla\times (\H_k-\H_0)\|_{L^2(\O)}+\|(\sigma(w_k)^{-1}-\sigma(w_0)^{-1})\nabla\times\H_0\|_{L^2(\O)}\\
\leq& \left({\sigma_2\over \sigma_1}+1\right)\|(\sigma(w_k)^{-1}-\sigma(w_0)^{-1})\nabla\times\H_0\|_{L^2(\O)}\rightarrow 0.
\endaligned
\end{equation}

As in \eqref{decomh} we have
$$\sigma(w_k)^{-1}\nabla\times\H_k-\E^0=\nabla\varphi_k+\h_k,\q \sigma(w_0)^{-1}\nabla\times\H_0-\E^0=\nabla\varphi_0+\h_0.
$$
Using the formula \eqref{formula-cj} for the representations of $\h_k$ and $\h_0$, and using \eqref{converg2} we obtain
\begin{equation}\label{converg3}
\aligned
\|\h_k-\h_0\|_{L^3(\O)}&
=\Big\|\sum_{j=1}^m\Big\{\int_\O\e_j\cdot[\sigma(w_k)^{-1}\nabla\times\H_k-\sigma(w_0)^{-1}\nabla\times\H_0]dx\Big\}\e_j\Big\|_{L^3(\O)}\\
&\leq\|\sigma(w_k)^{-1}\nabla\times\H_k-\sigma(w_0)^{-1}\nabla\times\H_0\|_{L^2(\O)}\sum_{j=1}^m\|\e_j\|_{L^2(\O)}\|\e_j\|_{L^3(\O)}\\
&\leq C(\O)\|\sigma(w_k)^{-1}\nabla\times\H_k-\sigma(w_0)^{-1}\nabla\times\H_0\|_{L^2(\O)}\rightarrow 0.
\endaligned
\end{equation}
It follows that
\begin{equation*}
\aligned
\|\nabla\varphi_k-\nabla\varphi_0\|_{L^2(\O)}&=\|\sigma(w_k)^{-1}\nabla\times\H_k-\h_k-\sigma(w_0)^{-1}\nabla\times\H_0+\h_0\|_{L^2(\O)}\\
&\leq \|\sigma(w_k)^{-1}\nabla\times\H_k-\sigma(w_0)^{-1}\nabla\times\H_0\|_{L^2(\O)}+\|\h_k-\h_0\|_{L^2(\O)}\rightarrow 0.
\endaligned
\end{equation*}
Therefore, by Poincar\'{e} inequality, we have
$\|\varphi_k-\varphi_0\|_{L^2(\O)}\rightarrow 0.$
By the estimate \eqref{fai-c0a} we see that $\|\varphi_k\|_{C^{0,\alpha}(\overline{\O})}$ is bounded uniformly in $k$. By Arzela-Ascoli theorem we know that, for any sequence $k_j\to\infty$ there exist a subsequence $\{\varphi_{k_{j_l}}\}\subset C^0(\overline{\O})$ and $\varphi^*\in C^0(\overline{\O})$ such that $\varphi_{k_{j_l}}\rightarrow\varphi^*$ in $C^0(\overline{\O})$ as $k_{j_l}\rightarrow\infty$. By the above convergence in $L^2(\O)$, we obtain $\varphi^*=\varphi_0$. Thanks to the uniqueness of $\varphi^*$, we have
\begin{equation}\label{converg4}
\|\varphi_k-\varphi_0\|_{C^0(\overline{\O})}\rightarrow 0 \text{ as $k\rightarrow\infty$.}
\end{equation}

By subtraction of the equations of $u_k$ and $u_0$, we get an equation for $u_k-u_0$ in $\O$:
$$
-\Delta (u_k-u_0)=\nabla\cdot(\varphi_k\,\nabla\times\H_k-\varphi_0\nabla\times\H_0)+(\h_k+\E^0)\cdot\nabla\times\H_k-(\h_0+\E^0)\cdot\nabla\times\H_0.
$$
We re-collect the two terms in the right side as follows:
$$
\varphi_k\,\nabla\times\H_k-\varphi_0\nabla\times\H_0=\varphi_k(\nabla\times\H_k-\nabla\times\H_0)+(\varphi_k-\varphi_0)\nabla\times\H_0,
$$
and
$$\aligned
&(\h_k+\E^0)\cdot\nabla\times\H_k-(\h_0+\E^0)\cdot\nabla\times\H_0=(\h_k+\E^0)\cdot(\nabla\times\H_k-\nabla\times\H_0)\\
&\q+(\h_k-\h_0)\cdot\nabla\times\H_0.
\endaligned
$$
Recall that $u_k-u_0=0$ on $\p\O$. Applying the $H^1$ estimate of Laplace equation to the above equation for $u_k-u_0$
we have
$$\aligned
&\|\nabla u_k-\nabla u_0\|_{L^2(\O)}\leq\|\varphi_k\,\nabla\times\H_k-\varphi_0\nabla\times\H_0\|_{L^2(\O)}\\
+&C\|(\h_k+\E^0)\cdot\nabla\times\H_k-(\h_0+\E^0)\cdot\nabla\times\H_0\|_{L^\frac{6}{5}(\O)}\\
\leq& \|\varphi_k\|_{L^\infty(\O)}\|\nabla\times(\H_k-\H_0)\|_{L^2(\O)}+\|\varphi_k-\varphi_0\|_{L^\infty(\O)}\|\nabla\times\H_0\|_{L^2(\O)}\\
+&C\{\|\E^0\|_{L^3(\O)}\|\nabla\times(\H_k-\H_0)\|_{L^2(\O)}+\|\h_k-\h_0\|_{L^3(\O)}\|\nabla\times\H_0\|_{L^2(\O)}\}\rightarrow 0,
\endaligned
$$
where $C$ depends on $\O$. Here we have used \eqref{converg1}, \eqref{converg3}, \eqref{converg4}. Hence by Poincar\'{e} inequality we find
$u_k\to u_0$ in $H^1(\O)$ as $k\to\infty.$ So $\mathrm{T}$ is
continuous from $L^2(\O)$ to $H^1(\O)$.

Finally, since the embedding $H^1(\O)\hookrightarrow L^2(\O)$ is compact, $\mathrm{T}$ is compact on $D$. Applying Schauder's fixed point theorem we conclude that $\mathrm{T}$ has a fixed point $u\in D$. Since $\mathrm{T}$ maps $D$ into $H^1(\O)$ we know that $u=\mathrm{T}(u)\in H^1(\O)$. Let $\H\in H_0(\divg0,\O)\cap H(\curl,\O)\cap\Bbb H_N(\O)^\perp$ be the solution of \eqref{equation-H} with $w$ replaced by $u$. Then $(u,\H)$ is a weak solution of \eqref{PZ0}.
\end{proof}

\v0.1in

\section{Regularity of Weak Solutions}\label{sec4}

\subsection{Higher integrability of derivatives}\

In Theorem \ref{PZ1} we get a weak solution $(u,\H)$ to \eqref{PZ0}. Under the assumption that $\O$ is of class $C^2$, we can show that actually $\H\in W^{1,p}(\O,\Bbb R^3)$ whenever $\E^0\in L^p(\O,\Bbb R^3)$, where $p$ is either slightly larger than $2$ (see Proposition \ref{Prop4.1}), or $p>3$ (see Theorem \ref{Thm-Lp-reg}).

\begin{Prop}\label{Prop4.1}
Assume $\O$ is a bounded domain in $\Bbb R^3$ with a $C^2$ boundary, and the function $\sigma$ satisfies \eqref{cond-rho}. Let $(u,\H)\in H^1(\O)\times [H_0(\divg0,\O)\cap H(\curl,\O)]$ be a weak solution of \eqref{PZ0} corresponding to the boundary datum $(u^0,\E^0)$.  Then there exists a constant $p_0>2$, which depends only on $\O,\sigma_1,\sigma_2$, such that the following conclusions hold.
\begin{itemize}
\item[(i)] If $\E^0\in L^p(\O,\Bbb R^3)$  with $2<p<p_0$, then $u\in W^{2,p/2}_{\loc}(\O)$, $\H\in W^{1,p}(\O,\Bbb R^3)$ and
\eq\label{est-w1p-H}
\|\H\|_{W^{1,p}(\O)}\leq  C_1\{\|\H\|_{L^2(\O)}+\|\E^0\|_{L^p(\O)}\}.
\eeq
The term $\|\H\|_{L^2(\O)}$ in the right side of  \eqref{est-w1p-H} can be removed if we choose $\H\in \Bbb H_N(\O)^\perp$.
\item[(ii)] If furthermore $u^0\in W^{2,p/2}(\O)$, then
$u\in W^{2,p/2}(\O)$, and we have
\eq\label{est-w1p-u}
\|u\|_{W^{2,p/2}(\O)}\leq C_2\{\|\E^0\|_{L^p(\O)}^2+\|u^0\|_{W^{2,p/2}(\O)}\}.
\eeq
\end{itemize}
In the above, $C_1, C_2$ depend only on $\O, p, \sigma_1, \sigma_2$.
\end{Prop}

\begin{proof} We first mention that, if $(u,\H)$ is a weak solution of \eqref{PZ0}, then for any $\h\in \Bbb H_N(\O)$, $(u,\H+\h)$ is also a weak solution of \eqref{PZ0}. Since $\Bbb H_N(\O)$ is of finite dimension,  we can always choose $\h\in \Bbb H_N(\O)$ such that $\H+\h\in \Bbb H_N(\O)^\perp$.

 {\it Step 1}.
Since
$$\aligned
&\nabla\times[\sigma(u)^{-1}\nabla\times\H-\E^0]=\0\q\text{in }\O,\\
&\nu\times[\sigma(u)^{-1}\nabla\times\H-\E^0]=\0\q\text{on }\p\O,
\endaligned
$$
by Lemma \ref{curl0}, there exist $\varphi\in H_0^1(\O)$ and $\h\in \Bbb H_D(\O)$ such that
\eqref{decom-3.1} holds.
Write
$$\h=\sum_{j=1}^mc_j\e_j,\q
c_j=\int_\O\e_j\cdot(\sigma(u)^{-1}\nabla\times\H-\E^0)dx,\q j=1,\cdots, m.
$$
Since $\O$ is of class $C^2$, we have $\Bbb H_D(\O)\subset C^0(\overline{\O},\Bbb R^3)$ and
\eq\label{Lp-h}
\|\h\|_{L^p(\O)}\leq C_3\|\E^0\|_{L^2(\O)},
\eeq
where $C_3$ depends on $\O, p, \sigma_1, \sigma_2$.
\v0.05in

{\it Step 2.}
By a similar derivation used for \eqref{equation-fai}, we see that $\varphi$ is a weak solution of
\begin{equation}\label{eq-Dir}
\left\{\aligned
&\nabla\cdot[\sigma(u)(\nabla\varphi+\h+\E^0)]=0\q &\text{in } \O,\\
&\varphi=0\q &\text{on } \p\O.
\endaligned\right.
\end{equation}
We show that there exists $p_0>2$ which depends only on $\O, \sigma_1, \sigma_2$, but is independent of the solution, such that for any $2<p<p_0$ and for any weak solution $\varphi$ of \eqref{eq-Dir}, it holds that
\eq\label{est-4.5}
\|\nabla\varphi\|_{L^p(\O)}\leq C_4\|\E^0\|_{L^p(\O)},
\eeq
where $C_4$ depends on $\O, p, \sigma_1, \sigma_2$.

In fact, by Meyers' estimate of higher integrability of gradient (see \cite{Meyers1963}), there exists $p_0>2$, which depends only on $\O,\sigma_1,\sigma_2$, but it is independent of the solution $\varphi$ of \eqref{eq-Dir}, such that $\varphi\in W^{1,p}(\O)$ for any $2<p< p_0$. Moreover, we have the estimate
$$
\|\nabla\varphi\|_{L^p(\O)}\leq C(\O,p,\sigma_1, \sigma_2)\{\|\h\|_{L^p(\O)}+\|\E^0\|_{L^p(\O)}\}.
$$
From this and using  \eqref{Lp-h} we get \eqref{est-4.5}.

\v0.05in

{\it Step 3}. We prove $\H\in W^{1,p}(\O,\Bbb R^3)$.

First, from \eqref{Lp-h} and \eqref{est-4.5} we see that $\nabla\times \H\in L^p(\O,\Bbb R^3)$ and
\eq\label{est-curl-Hp}
\|\nabla\times\H\|_{L^p(\O)}\leq C_5\|\E^0\|_{L^p(\O)},
\eeq
where $C_5$ depends on $\O, p, \sigma_1, \sigma_2$.

Next, we show $\H\in L^p(\O,\Bbb R^3)$. We prove this by a duality method, which has been used in the proof of \cite[Lemma 3.1]{XZ2015}.
Given $\F\in C_c^\infty(\O,\Bbb R^3)$ and $1<r<\infty$, by Helmholtz-Weyl decomposition (see \cite[Theorem 6.1]{AS2013} or \cite[Theorem 2.1]{KY2009}), there exist $\w\in W^{1,r}(\O,\Bbb R^3)$ with $\nu\times\w=\0$ on $\p\O$, $\chi\in  W^{1,r}(\O)$,  and $\z\in \Bbb H_N(\O)$ such that
$$\F=\nabla\times \w+\nabla \chi+\z.
$$
Moreover, the triplet $(\w,\chi,\z)$ satisfies the estimate
\begin{equation}\label{est-dec}
\|\w\|_{W^{1,r}(\O)}+\|\chi\|_{W^{1,r}(\O)}+\|\z\|_{L^\infty(\O)}\leq C(\O,r)\|\F\|_{L^r(\O)}.
\end{equation}

Using \eqref{decom-3.1} we see that $\H$ satisfies the following div-curl system
\begin{equation}\label{div-curl-H}
\left\{\aligned
&\nabla\times\H=\sigma(u)(\nabla\varphi+\h+\E^0),\q \nabla\cdot \H=0\q&\text{in }\O,\\
&\nu\cdot\H=0\q&\text{on }\p\O,
\endaligned\right.
\end{equation}
where $\sigma(u)(\nabla\varphi+\h+\E^0)\in L^p(\O,\Bbb R^3)$. Since $\nu\times\w=\0$ on $\p\O$, $\nabla\cdot \H=0$ in $\O$ and $\nu\cdot\H=0$ on $\p\O$, we have
$$\aligned
&\int_\O  \H\cdot(\nabla\times \w) dx=\int_\O (\nabla\times\H)\cdot\w dx+\int_{\p\O}(\nu\times\w)\cdot\H dS=\int_\O\nabla\times\H\cdot\w dx,\\
&\int_\O \H\cdot\nabla \chi dx=\int_{\p\O} (\nu\cdot\H) \chi dS-\int_\O  (\nabla\cdot\H)\chi dx=0.
\endaligned
$$
Using the above two equalities and \eqref{est-dec}, we get
$$
\aligned
\int_\O \H\cdot \F dx&=\int_\O \H\cdot(\nabla\times \w+\nabla \chi+\z)dx\\
&=\int_\O (\nabla\times\H)\cdot\w+\H\cdot \z dx\\
&\leq C(\O,p)\{\|\nabla\times\H\|_{L^p(\O)}+
\|\H\|_{L^2(\O)}\}\|\F\|_{L^{p'}(\O)},
\endaligned
$$
from which we obtain $\H\in L^p(\O,\Bbb R^3)$ with the estimate
\eq\label{H-lp}
\|\H\|_{L^{p}(\O)}\leq C(\O,p)\{\|\H\|_{L^2(\O)}+\|\nabla\times\H\|_{L^p(\O)}\}.
\eeq

With $\H\in L^p(\O,\Bbb R^3)$ in hand, we can apply the $L^p$ regularity theory for the div-curl systems (see  \cite[Theorem 2.2]{AS2011a} and \cite[Theorem 3.5]{AS2013}; see also \cite{vonW1992} and \cite{KY2009}) to \eqref{div-curl-H}, and conclude that $\H\in W^{1,p}(\O,\Bbb R^3)$. Since $\div\H=0$ in $\O$ and $\nu\cdot\H=0$ on $\p\O$, we have
$$
\|\H\|_{W^{1,p}(\O)}\leq C(\O,p)\{\|\H\|_{L^p(\O)}+\|\nabla\times\H\|_{L^p(\O)}\}.
$$
From this, \eqref{est-curl-Hp} and \eqref{H-lp},  we get \eqref{est-w1p-H}.

If we choose $\H\in \Bbb H_N(\O)^\perp$, then, $\div\H=0$ in $\O$, and $\nu\cdot\H=0$ on $\p\O$, so we can use the following Poincar\'{e} type inequality
$$
\|\H\|_{L^2(\O)}\leq C(\O)\|\nabla\times\H\|_{L^2(\O)}.
$$
From this and \eqref{M-C1},  by increasing the constant $C_1$  if necessary, we can remove the term $\|\H\|_{L^2(\O)}$ in the right side of \eqref{est-w1p-H}.

\v0.05in

{\it Step 4}. Finally, using \eqref{cond-rho} we see that $\sigma(u)^{-1}|\nabla\times\H|^2\in L^{p/2}(\O)$.
Applying elliptic regularity theory to the Laplace equation \eqref{eq-u}, we see that $u\in W^{2,p/2}_{\loc}(\O)$. If furthermore $u^0\in W^{2,p/2}(\O)$, then we have $u\in W^{2,p/2}(\O)$, and
$$\|u\|_{W^{2,p/2}(\O)}\leq C(\O,p)\{\|\sigma(u)^{-1}|\nabla\times\H|^2\|_{L^{p/2}(\O)}+\|u^0\|_{W^{2,p/2}(\O)}\}.
$$
From this and \eqref{est-curl-Hp} we get \eqref{est-w1p-u}.
\end{proof}



Now we show that if $(u^0,\E^0)$ satisfies
\eq\label{u0H0w1q}
(u^0,\E^0)\in W^{1,q}(\O)\times L^q(\O,\Bbb R^3)\q\text{for some }q>3,
\eeq
then the weak solution of \eqref{PZ0} has $W^{1,q}$ regularity.

\begin{Thm}\label{Thm-Lp-reg}
Assume that $\O$ is a bounded domain in $\Bbb R^3$ with a $C^2$ boundary, the function $\sigma$ satisfies \eqref{cond-rho}, and
$(u^0,\E^0)$ satisfies \eqref{u0H0w1q}. Let $(u,\H)\in H^1(\O)\times [H_0(\divg0,\O)\cap H(\curl,\O)]$ be a weak solution of \eqref{PZ0}. Then we have the following conclusions.
\begin{itemize}
\item[(i)] $u\in C^{0,(\mu-1)/2}(\overline{\O})$ for all $1<\mu<1+2\min\{\delta,1-3/q\}$, where $\delta=\delta(\O,\sigma_1,\sigma_2)\in (0,1)$, and there exists $C_1=C_1(\O,\mu,q,\sigma_1,\sigma_2)>0$ such that
\eq
\|u\|_{C^{0,(\mu-1)/2}(\overline{\O})}\leq C_1\{\|\E^0\|_{L^q(\O)}^2+\|u^0\|_{W^{1,q}(\O)}\}.
\eeq
\item[(ii)] $(u,\H)\in W^{1,q}(\O)\times W^{1,q}(\O,\Bbb R^3)$\footnote{Then by Morrey embedding theorem,
$(u,\H)\in C^{0, 1-3/q}(\overline{\O})\times C^{0,1-3/q}(\overline{\O},\Bbb R^3).$}, and
\eq\label{est-w1p-H-2}
\|\H\|_{W^{1,q}(\O)}\leq  C_2\{\|\H\|_{L^2(\O)}+\|\E^0\|_{L^q(\O)}\},
\eeq
where $C_2$ depends on $\O,q,\sigma_1,\sigma_2$ and  the VMO modulus of continuity of $\sigma(u)$.
If furthermore we choose $\H\in \Bbb H_N(\O)^\perp$, then the term $\|\H\|_{L^2(\O)}$ in the right side of  \eqref{est-w1p-H-2} can be removed.
\item[(iii)]  Assume furthermore the function $\sigma$ satisfies \eqref{Lip-L}.
Then we have the estimate
\eq\label{curl-key}
\|\nabla\times\H\|_{L^q(\O)}\leq  C_3\|\E^0\|_{L^q(\O)},
\eeq
where the constant $C_3$ depends only on $\O,q,\sigma_1,\sigma_2,L$, and the $W^{1,q}(\O)$ norm of $u^0$ and the $L^q(\O)$ norm of $\E^0$.
\end{itemize}
\end{Thm}

\begin{proof}
The key point in the proof is to establish  first the H\"{o}lder continuity of $u$.

{\it Step 1}.
Let $(u,\H)$ be a weak solution of \eqref{PZ0}.
As in the proof of Theorem \ref{PZ1}, we have the equality \eqref{decom-3.1},
where $\h\in \Bbb H_D(\O)$, $\H$, $\varphi$ and $u$ are solutions of \eqref{equation-H}, \eqref{equation-fai} and \eqref{equation-u}, with $w$ replaced by $u$, respectively.
By Lemma \ref{lemma3.3}, we have the $L^2$ estimate for $\nabla\times\H$:
$$\|\nabla\times\H\|_{L^2(\O)}\leq\sigma_2\|\E^0\|_{L^2(\O)},$$
and the $L^2$ estimate for $u$:
\begin{equation}\label{R1}
\|u\|_{L^2(\O)}\leq C_4\{\|\E^0\|_{L^q(\O)}\|\E^0\|_{L^2(\O)}+\|u^0\|_{H^1(\O)}\},
\end{equation}
where $C_4$ depends on $\O,q,\sigma_1,\sigma_2$.

Next we show the inequality
\begin{equation}\label{R2}
\|\varphi\|_{C^{0,(\mu-1)/2}(\overline{\O})}+\|\nabla\varphi\|_{L^{2,\mu}(\O)}\leq C_5\|\E^0\|_{L^q(\O)},
\end{equation}
where  $1<\mu<1+2\min\{\delta,1-3/q\}$, the constants $\delta=\delta(\O,\sigma_1,\sigma_2)\in (0,1)$, and $C_5=C_5(\O,\mu,q,\sigma_1,\sigma_2)$.  To prove this conclusion, we apply Lemma \ref{M-Campanato3} to \eqref{equation-fai} and conclude that, there exists $\delta=\delta(\O,\sigma_1,\sigma_2)\in (0,1)$, such that the following estimate holds for all $0<\mu<1+2\delta$:
\begin{equation}\label{M-C-fai}
\aligned
\|\nabla\varphi\|_{L^{2,\mu}(\O)}&\leq C\{\|\varphi\|_{H^1(\O)}+\|\sigma(u)(\h+\E^0)\|_{L^{2,\mu}(\O)}\}\\
&\leq C\{\|\E^0\|_{L^2(\O)}+\|\h\|_{L^{2,\mu}(\O)}+\|\E^0\|_{L^{2,\mu}(\O)}\}.
\endaligned
\end{equation}
Here we have used \eqref{M-C1} and the standard $H^1$ estimate for \eqref{eq-phi}, and the constant $C$ depends on $\O,\mu,\sigma_1,\sigma_2$.
Then by Lemma \ref{M-Campanato} we obtain
\begin{equation}\label{M-C2}
\aligned
\|\varphi\|_{L^{2,\mu+2}(\O)}&\leq C(\O,\mu)\{\|\varphi\|_{L^2(\O)}+\|\nabla\varphi\|_{L^{2,\mu}(\O)}\}\\
&\leq C_6\{\|\E^0\|_{L^2(\O)}+\|\h\|_{L^{2,\mu}(\O)}+\|\E^0\|_{L^{2,\mu}(\O)}\},
\endaligned
\end{equation}
where the constant $C_6$ depends on $\O,\mu,\sigma_1,\sigma_2$.

Since $\O$ is of class $C^2$, $\mathbb{H}_D(\O)\subset C^0(\overline{\O},\Bbb R^3)$. Recalling that $\h$ is in the form of \eqref{formula-cj}, we get
\begin{equation}\label{est-infty}
\|\h\|_{L^\infty(\O)}\leq C(\O,\sigma_1,\sigma_2)\|\E^0\|_{L^2(\O)}.
\end{equation}
Using the embedding of $L^r(\O)$ into a Campanato space, we have
$$L^{r}(\O)\subseteq L^{2, 3-6/r}(\O)\subseteq L^{2,\mu}(\O)\q\text{for }r>2,\;\; 0\leq \mu\leq 3-6/r.
$$
So we get
\begin{equation}\label{M-C3}
\aligned
\|\h\|_{L^{2,\mu}(\O)}+\|\E^0\|_{L^{2,\mu}(\O)}&\leq C(\O,\mu,q)\{\|\h\|_{L^{q}(\O)}+\|\E^0\|_{L^{q}(\O)}\}\\
&\leq C(\O,\mu,q)\{\|\h\|_{L^\infty(\O)}+\|\E^0\|_{L^{q}(\O)}\},
\endaligned
\end{equation}
where $1<\mu<1+2\min\{\delta,1-3/q\}$. Then by Lemma \ref{M-Campanato}, inequalities \eqref{M-C-fai}, \eqref{M-C2}, \eqref{est-infty} and
\eqref{M-C3} we get \eqref{R2}.

\v0.05in

{\it Step 2}.
Since $\nabla\times\H=\sigma(u)(\nabla\varphi+\h+\E^0)$, by \eqref{R2}, \eqref{M-C3} and \eqref{est-infty} we get
\begin{equation}\label{R3}
\|\nabla\times\H\|_{L^{2,\mu}(\O)}\leq C(\O,\mu,q,\sigma_1,\sigma_2)\|\E^0\|_{L^q(\O)}.
\end{equation}
Since $\nabla\times\E^0=0$ in $\O$, by Lemma \ref{curl0}, there exist $\varphi^0\in \dot{H}^1(\O)$ and $\h_1\in \Bbb H_N(\O)$ such that $\E^0=\nabla\varphi^0+\h_1$. Here $\varphi^0$ satisfies that $\Delta\varphi^0=\nabla\cdot \E^0$ in $\O$ and $\nu\cdot\nabla\varphi^0=\nu\cdot\E^0$ on $\p\O$.
Hence $\varphi^0\in W^{1,q}(\O)$. Similar to \eqref{est-infty}, we also have the $L^\infty$ estimate for $\h_1$:
\begin{equation}\label{est-infty1}
\|\h_1\|_{L^\infty(\O)}\leq C(\O)\|\E^0\|_{L^2(\O)}.
\end{equation}
Since
$$
\aligned
\sigma(u)^{-1}|\nabla\times\H|^2&=(\nabla\varphi+\h+\E^0)\cdot\nabla\times\H=[\nabla(\varphi+\varphi^0)+\h+\h_1]\cdot\nabla\times\H\\
&=\nabla\cdot[(\varphi+\varphi^0)\nabla\times\H]+(\h+\h_1)\cdot\nabla\times\H,
\endaligned
$$
we can write the equation for $u$ in the following form:
\begin{equation}\label{equ-6}
-\Delta u=\nabla\cdot[(\varphi+\varphi^0)\nabla\times\H]+(\h+\h_1)\cdot\nabla\times\H \q\text{in } \O,\q
u=u^0\q\text{ on } \p\O.
\end{equation}
Applying Lemma \ref{M-Campanato3} to the above equation, and using Lemma \ref{M-Campanato} we obtain
$$
\aligned
\|\nabla u\|_{L^{2,\mu}(\O)}&\leq C\{\|u\|_{H^1(\O)}+\|(\h+\h_1)\cdot\nabla\times\H\|_{L^{2,(\mu-2)^+}(\O)}\\
&+\|(\varphi+\varphi^0)\nabla\times\H\|_{L^{2,\mu}(\O)}+\|\nabla u^0\|_{L^{2,\mu}(\O)}\}\\
&\leq C_7\{(\|\h+\h_1\|_{L^\infty(\O)}+\|\varphi+\varphi^0\|_{L^\infty(\O)})\|\nabla\times\H\|_{L^{2,\mu}(\O)}\\
&+\|\E^0\|_{L^q(\O)}\|\E^0\|_{L^2(\O)}+\|u^0\|_{H^1(\O)}+\|\nabla u^0\|_{L^{2,\mu}(\O)}\},
\endaligned
$$
where $C_7$ depends on $\O,\mu,q,\sigma_1,\sigma_2$. Then by Lemma \ref{M-Campanato}, inequalities \eqref{R1}, \eqref{R2}, \eqref{est-infty}, \eqref{R3}, \eqref{est-infty1},  we get
\begin{equation*}
\aligned
\|u\|_{C^{0,(\mu-1)/2}(\overline{\O})}&\leq C\|u\|_{L^{2,\mu+2}(\O)}
\leq C\{\|u\|_{L^2(\O)}+\|\nabla u\|_{L^{2,\mu}(\O)}\}\\
&\leq C_8\{\|\E^0\|_{L^q(\O)}^2+\|u^0\|_{W^{1,q}(\O)}\},\\
\endaligned
\end{equation*}
where $C_8$ depends on $\O,\mu,q,\sigma_1,\sigma_2$.

\v0.05in

{\it Step 3}. Now $u$ is continuous on $\bar\O$. By the continuity of the function $\s$ we see that $\s(u)$ is continuous, hence $\s(u)\in VMO(\O)$.
Therefore we can apply \cite[Theorem 1]{AQ2002} to Dirichlet problem \eqref{eq-Dir}, and get $\varphi\in W^{1,q}(\O)$ with the estimate
$$
\|\nabla\varphi\|_{L^q(\O)}\leq C_9\|\sigma(u)(\h+\E^0)\|_{L^q(\O)},
$$
where $C_9$ depends on $\O,q,\sigma_1,\sigma_2$ and  the VMO modulus of continuity of $\sigma(u)$. Hence
$$\nabla\times \H=\sigma(u)(\nabla\varphi+\h+\E^0)\in L^q(\O,\Bbb R^3).
$$
Then, in the same way as in Step 3 of the proof of Proposition \ref{Prop4.1}, we conclude that $\H\in W^{1,q}(\O,\Bbb R^3)$, and we also have the estimate
$$
\|\H\|_{W^{1,q}(\O)}\leq C_{10}\{\|\H\|_{L^2(\O)}+\|\E^0\|_{L^q(\O)}\},
$$
where the constant $C_{10}$, differently from the case of Proposition \ref{Prop4.1}, depends not only $\O, q, \sigma_1, \sigma_2$, but also on  the VMO modulus of continuity of $\sigma(u)$.

Next we re-write the equation for $u$ in the following form
\eq\label{equ-5}
-\Delta u=\nabla\cdot[\H\times\sigma(u)^{-1}\nabla\times \H]\q\text{in }\O,\q
 u=u^0 \text{ on } \p\O.
\eeq
Since $q>3$, by the Sobolev embedding theorem we see that $\H\in C^0(\overline\O,\Bbb R^3)$. Hence
$$\aligned
&\nabla\cdot(\H\times\sigma(u)^{-1}\nabla\times \H)\in W^{-1,q}(\O).
\endaligned
$$
By elliptic regularity theory, we obtain $u\in W^{1,q}(\O)$.
We can also obtain an estimate of $\|u\|_{W^{1,q}(\O)}$, with the constant depending also on  the VMO modulus of continuity of $\sigma(u)$.

{\it Step 4}. Assume $\sigma$ satisfies \eqref{Lip-L}. Then $\s(u)\in C^{0,(\mu-1)/2}(\overline{\O})$. Applying \cite[Theorem 3.16 (iv)]{Tro1987} to Dirichlet problem \eqref{eq-Dir}, we obtain
$$
\aligned
\|\nabla\varphi\|_{L^q(\O)}&\leq C\|\sigma(u)(\h+\E^0)\|_{L^q(\O)}\leq C_{11}\|\E^0\|_{L^q(\O)},
\endaligned
$$
where $C_{11}$ depends on $\O,q,\sigma_1,\sigma_2$, and also on the bound of $\|\s(u)\|_{C^{0,(\mu-1)/2}(\overline{\O})}$, which can be estimated as follows:
$$
\aligned
\|\s(u)\|_{C^{0,(\mu-1)/2}(\overline{\O})}&\leq\sigma_2+L\|u\|_{C^{0,(\mu-1)/2}(\overline{\O})}\\
&\leq\sigma_2+LC_8\Big\{\|\E^0\|_{L^q(\O)}^2+\|u^0\|_{W^{1,q}(\O)}\Big\}.
\endaligned
$$
So $C_{11}$ depends only on $\O,q,\sigma_1,\sigma_2,L$ and the norms of $\E^0$ and $u^0$ appearing in the above inequality. Thus we get (iii).
\end{proof}

\begin{Rem}\label{Rem-key}
Let $\|\E^0\|_{L^q(\O)}\leq 1$. Then
$$
\|\s(u)\|_{C^{0,(\mu-1)/2}(\overline{\O})}\leq\sigma_2+LC_8\Big\{1+\|u^0\|_{W^{1,q}(\O)}\Big\}.
$$
Hence the constant $C_3$ in \eqref{curl-key} depends only on $\O,q,\sigma_1,\sigma_2,L$, $\|u^0\|_{W^{1,q}(\O)}$. This point will play an important role in the proof of small boundary data uniqueness in section 5.
\end{Rem}

\v0.05in

\subsection{H\"older continuity of derivatives}\

Next we prove H\"older continuity of derivatives of $u$ and $\H$ as the domain and boundary data allow.

\begin{Thm}\label{Thm-Holder}
Let $k\geq 0$ be an integer, and $\a\in(0,1)$. Let $\O$ be a bounded and simply connected $C^{k+2,\a}$ domain in $\Bbb R^3$, and $\sigma$ satisfy \eqref{cond-rho}. Let $(u,\H)\in H^1(\O)\times [H_0(\divg0,\O)\cap H(\curl,\O)]$ be a weak solution of \eqref{PZ0}. Assume in addition that
\eq\label{Holder-bdry}
 (u^0, \E^0)\in C^{k+1,\a}(\overline{\O})\times C^{k,\a}(\overline{\O},\Bbb R^3),\q \sigma\in C^{k,\a}_{\loc}(\Bbb R),
\eeq
then we have $(u, \H)\in C^{k+1,\a}(\overline{\O})\times C^{k+1,\a}(\overline{\O},\Bbb R^3)$.
\end{Thm}

\begin{proof} We give the proof for $k=0, 1$. Then by induction we get the conclusion for $k\geq 2$.
\vskip0.05in
{\it Step 1}.
We start with the case where $k=0$. Suppose $(u^0, \E^0)\in C^{1,\a}(\overline{\O})\times C^{0,\a}(\overline{\O},\Bbb R^3)$ and $\sigma\in C^{0,\a}_{\loc}(\Bbb R)$. By Theorem \ref{Thm-Lp-reg} we obtain $\sigma(u)\in C^{0,\beta}(\overline{\O})$, where $\beta=(\mu-1)\a/2$. Applying elliptic regularity theory to the equation \eqref{eq-Dir}, we conclude that
$\varphi\in C^{1,\beta}(\overline{\O}).$
Noting the regularity of the elements in $\Bbb H_D(\O)$, using the identity
$$\sigma(u)^{-1}|\nabla\times \H|^2=\sigma(u)|\nabla\varphi+\h+\E^0|^2\in C^{0,\beta}(\overline{\O}),
$$
and applying elliptic regularity theory to Laplace equation \eqref{eq-u}, we conclude that $u\in C^{1,\a}(\overline{\O})$. It implies that $\sigma(u)\in C^{0,\a}(\overline{\O})$. Applying elliptic regularity theory to the equation \eqref{eq-Dir} again, we conclude that $\varphi\in C^{1,\a}(\overline{\O})$. Since $\nabla\times \H=\sigma(u)(\nabla\varphi+\h+\E^0)\in C^{0,\a}(\overline{\O},\Bbb R^3)$, using the H\"older regularity of the div-curl system \eqref{div-curl-H} given by \cite[Proposition 2.1]{BP2007}, we obtain $\H\in C^{1,\a}(\overline{\O},\Bbb R^3)$.

\vskip0.05in
{\it Step 2}. Now we consider the case where $k=1$.
Suppose $(u^0, \E^0)\in C^{2,\a}(\overline{\O})\times C^{1,\a}(\overline{\O},\Bbb R^3)$ and $\sigma\in C^{1,\a}_{\loc}(\Bbb R)$. Applying the H\"older regularity of Laplace equation to \eqref{eq-u} we derive $u\in C^{2,\a}(\overline{\O})$. Using \eqref{eq-Dir} we obtain $\varphi\in C^{2,\a}(\overline{\O})$. Since $\sigma(u)\nabla\varphi\in C^{1,\a}(\overline{\O})$, applying \cite[Proposition 2.1]{BP2007} to \eqref{div-curl-H} we have $\H\in C^{2,\a}(\overline{\O},\Bbb R^3)$.
\end{proof}
Here we mention that the local Schauder theory has been given by Kang and Kim, see \cite[Theorem 3.2 and Remark 3.3]{KK2002} or \cite[Remark 5.10]{KK2011}.

\vskip0.1in

\section{Uniqueness under Small Boundary Data}\label{sec5}

In this section we establish uniqueness results under small boundary data. We assume that the function $\sigma$ satisfies \eqref{Lip-L}.
Let $S(n,p,\O)$ be the best constant for Sobolev inequality
$$
S(n,p,\O)\|v\|_{L^{2p/(p-2)}(\O)}\leq \|\nabla v\|_{L^2(\O)}\text{ for any $v\in H_0^1(\O)$.}
$$
It is well-known that if $p=n>2$, then $S(n,p,\O)$ does not depend on $\O$. So we denote $S(3,3,\O)$ by $S(3)$.

In order to prove the uniqueness result for \eqref{PZ0}, we establish the following lemma, which is similar to \cite[Theorem 5]{HRS1993}.

\begin{Lem}\label{unique}
Let $\O$ be a bounded domain in $\Bbb R^3$, and $\sigma$ satisfy \eqref{cond-rho} and \eqref{Lip-L}.
Let $\kappa$ be a positive constant satisfying
$$\kappa<\frac{S(3)\sigma_1}{\sqrt{(2\sigma_2/\sigma_1+1)L}}.
$$
Then \eqref{PZ0} has at most one weak solution lying in the following set
$$\left\{(\u,\H)\in H^1(\O)\times [H_0(\divg0,\O)\cap H(\curl,\O)\cap\Bbb H_N(\O)^\perp]:\|\nabla\times\H \|_{L^3(\O)}\leq \kappa\right\}.$$
\end{Lem}

\begin{proof}
Let $(u_1,\H_1)$ and $(u_2,\H_2)$ be two weak solutions in the above set. Set $v=u_1-u_2$ and $\B=\H_1-\H_2$. Then $v$ and $\B$ satisfy
\begin{equation*}
-\Delta v=\sigma(u_1)^{-1}|\nabla\times\H_1|^2-\sigma(u_2)^{-1}|\nabla\times\H_2|^2\q \text{in } \O,\q
 v=0\q \text{ on } \p\O,
\end{equation*}
and
\begin{equation*}
\left\{\aligned
&\nabla\times[\sigma(u_1)^{-1}\nabla\times \B]=\nabla\times[(\sigma(u_2)^{-1}-\sigma(u_1)^{-1}) \nabla\times\H_2],\q \nabla\cdot \B=0\q &\text{in } \O,\\
&\nu\cdot\B=0,\q \nu\times[\sigma(u_1)^{-1}\nabla\times \B]=\nu\times[(\sigma(u_2)^{-1}-\sigma(u_1)^{-1}) \nabla\times\H_2] &\text{ on } \p\O.
\endaligned\right.
\end{equation*}
Then we have the equalities
\begin{equation}\label{v-eq}
\int_\O|\nabla v|^2dx=\int_\O\{\sigma(u_1)^{-1}|\nabla\times\H_1|^2-\sigma(u_2)^{-1}|\nabla\times\H_2|^2\}v dx,
\end{equation}
\begin{equation}\label{B-eq}
\int_\O\sigma(u_1)^{-1}|\nabla\times \B|^2dx=\int_\O(\sigma(u_2)^{-1}-\sigma(u_1)^{-1}) \nabla\times\H_2\cdot \nabla\times \B dx.
\end{equation}
From \eqref{B-eq} and using the conditions \eqref{cond-rho} and \eqref{Lip-L}, we have
\begin{equation}\label{B-v}
\aligned
&\sigma_2^{-1}\|\nabla\times \B\|_{L^2(\O)}\leq \sigma_1^{-2}L\|v\nabla\times\H_2\|_{L^2(\O)}\\
\leq& \sigma_1^{-2}L\|v\|_{L^6(\O)}\|\nabla\times\H_2\|_{L^3(\O)}\leq \sigma_1^{-2}L\kappa\|v\|_{L^6(\O)}.
\endaligned
\end{equation}
Using \eqref{v-eq}, by Sobolev inequality and H\"{o}lder inequality, it follows that
\begin{equation}
\aligned
S(3)^2\|v\|^2_{L^6(\O)}\leq&\|\nabla v\|_{L^2(\O)}^2\\
\leq& \left\|\sigma(u_1)^{-1}|\nabla\times\H_1|^2-\sigma(u_2)^{-1}|\nabla\times\H_2|^2\right\|_{L^{6/5}(\O)}\|v\|_{L^6(\O)}\\
\leq& \|\sigma(u_1)^{-1}(\nabla\times \B)\cdot[\nabla\times (\H_1+\H_2)]\|_{L^{6/5}(\O)}\|v\|_{L^6(\O)}\\
&+\left\|(\sigma(u_1)^{-1}-\sigma(u_2)^{-1})|\nabla\times\H_2|^2\right\|_{L^{6/5}(\O)}\|v\|_{L^6(\O)}.
\endaligned
\end{equation}
By H\"{o}lder inequality and the conditions \eqref{cond-rho}, \eqref{Lip-L}, we get
\begin{equation}
\|\sigma(u_1)^{-1}\nabla\times \B\cdot\nabla\times (\H_1+\H_2)\|_{L^{6/5}(\O)}\leq 2\sigma_1^{-1}\kappa\|\nabla\times \B\|_{L^2(\O)},
\end{equation}
and
\begin{equation}
\aligned
&\left\|(\sigma(u_1)^{-1}-\sigma(u_2)^{-1})|\nabla\times\H_2|^2\right\|_{L^{6/5}(\O)}\leq \sigma_1^{-2}L\left\|v|\nabla\times\H_2|^2\right\|_{L^{6/5}(\O)}\\
\leq& \sigma_1^{-2}L\|v\|_{L^6(\O)}\left\||\nabla\times\H_2|^2\right\|_{L^{3/2}(\O)}
\leq \sigma_1^{-2}L\kappa^2\|v\|_{L^6(\O)}.
\endaligned
\end{equation}
Combining the above four inequalities, we obtain
$$S(3)^2\|v\|^2_{L^6(\O)}\leq \sigma_1^{-2}(\frac{2\sigma_2}{\sigma_1}+1)L\kappa^2\|v\|^2_{L^6(\O)}.$$
Hence, if $\sigma_1^{-2}(2\sigma_2/\sigma_1+1)L\kappa^2<S(3)^2$, then $v=0$. Consequently, \eqref{B-v} implies that $\nabla\times \B=\0$ in $\O$. Since $\nabla\cdot \B=0$ in $\O$, $\nu\cdot\B=0$ on $\p\O$ and $\B\in \Bbb H_N(\O)^\perp$, we derive $\B=\0$.
\end{proof}

Now we prove uniqueness of weak solutions of \eqref{PZ0} under small boundary data condition.

\begin{Thm}\label{Thm-uniq}
Let $\O$ be a bounded domain in $\Bbb R^3$ with a $C^2$ boundary and $\sigma$ satisfy \eqref{cond-rho} and \eqref{Lip-L}. Assume $(u^0,\E^0)$ satisfies \eqref{u0H0w1q}. Then there exists $\eta>0$ such that if \eqref{small-E0} holds for this $\eta$, then  \eqref{PZ0} has a unique weak solution in the space $H^1(\O)\times [H_0(\divg0,\O)\cap H(\curl,\O)\cap\Bbb H_N(\O)^\perp]$.
\end{Thm}

\begin{proof}
Theorem \ref{PZ1} proves the existence of at least one weak solution $(\tilde{u},\widetilde{\H})\in H^1(\O)\times [H_0(\divg0,\O)\cap H(\curl,\O)]$. Furthermore, we can choose $\widetilde{\H}\in \Bbb H_N(\O)^\perp$. And then we show the uniqueness.

We first assume \eqref{small-E0} holds with $\eta=1$. Let $(u,\H)$ be any possible weak solution of \eqref{PZ0}. By Remark \ref{Rem-key}, we conclude that the constant $C_3$ in \eqref{curl-key} depends only on $\O,q,\sigma_1,\sigma_2,L$, $\|u^0\|_{W^{1,q}(\O)}$. Hence
$$
\aligned
\|\nabla\times\H\|_{L^3(\O)}\leq C(\O,q)\|\nabla\times\H\|_{L^q(\O)}
\leq C_*\|\E^0\|_{L^q(\O)},
\endaligned
$$
where $C_*$ depends only on $\O,q,\sigma_1,\sigma_2,L,\|u^0\|_{W^{1,q}(\O)}$, and is independent of the solution. Let
$$\eta=\min \left\{1,\; \frac{S(3)\sigma_1}{2C_*\sqrt{(2\sigma_2/\sigma_1+1)L}}\right\}.
$$
If $\E^0$ satisfies \eqref{small-E0} for this $\eta$, then it holds that
$$\|\nabla\times\H \|_{L^3(\O)}\leq \frac{S(3)\sigma_1}{2\sqrt{(2\sigma_2/\sigma_1+1)L}},$$
 and hence uniqueness follows from Lemma \ref{unique}.
\end{proof}

\v0.1in

\section{Tangential boundary condition}
In this section, we establish existence, regularity and uniqueness of weak solutions of system \eqref{Pan-eq}.

\begin{Def}
We say that $(u,\H)\in H^1(\O)\times [H(\divg0,\O)\cap H(\curl,\O)]$  is a weak solution of \eqref{Pan-eq} if $u=u^0$ and $\nu\times\H=\nu\times\H^0$ on $\p\O$ in the sense of trace, and if it holds that
$$\aligned
&\int_\O\nabla u\cdot\nabla v\, dx=\int_\O\sigma(u)^{-1}|\nabla\times \H|^2v\, dx,\q&\forall v\in H^1_0(\O)\cap L^\infty(\O),\\
&\int_\O\sigma(u)^{-1}\nabla\times \H\cdot\nabla\times\w dx=0,\q &\forall  \w\in H(\divg0,\O)\cap H_0(\curl,\O).
\endaligned
$$
\end{Def}

\begin{Prop}[Existence of weak solutions]\label{PZ}
Let $\O$ be a bounded Lipschitz domain in $\Bbb R^3$. Assume the function $\sigma$ satisfies  \eqref{cond-rho}, $u^0\in H^1(\O)$ and $\H^0\in W^{1,q}(\O,\Bbb R^3)$ for some $q>3$.
Then \eqref{Pan-eq} has a weak solution $(u,\H)\in H^1(\O)\times [H(\divg0,\O)\cap H(\curl,\O)]$.
\end{Prop}

\begin{proof}
{\it Step 1}. For any given $w\in L^2(\O)$, let $\H_w\in H(\divg0,\O)\cap H(\curl,\O)$ be a weak solution of the system

\begin{equation}\label{eq3.28}
\left\{\aligned
&\nabla\times[\sigma(w)^{-1}\nabla\times \H_w]=\0,\q \nabla\cdot \H_w=0\q &\text{in } \O,\\
&\nu\times\H_w= \nu\times\H^0\q  &\text{on } \p\O.
\endaligned\right.
\end{equation}
Let $\phi\in H_0^1(\O)$ be such that
$\Delta\phi=\divg \H^0$ in $\O$ and $\phi=0$ on $\p\O$. Taking $\H_w-(\H^0-\nabla\phi)$ as a test function for \eqref{eq3.28}, we have the following $L^2$ estimate:
\eq\label{est3.29}
\|\nabla\times\H_w\|_{L^2(\O)}\leq {\sigma_2\over \sigma_1}\|\nabla\times\H^0\|_{L^2(\O)}.
\eeq

\vskip0.05in
{\it Step 2}.
Since $\nabla\times[\sigma(w)^{-1}\nabla\times\H_w]=\0\text{ in }\O,$
by Lemma \ref{curl0}, there exist $\varphi_w\in \dot{H}^1(\O)$ and $\h_w\in \Bbb H_N(\O)$ such that
\eq\label{decom-3.30}
\sigma(w)^{-1}\nabla\times\H_w=\nabla\varphi_w+\h_w.
\eeq
Let $\bold v_1,\cdots,\bold v_N$ be an orthonormal basis of $\Bbb H_N(\O)$ with respect to the $L^2$-norm. We can write
$$
\h_w=\sum_{j=1}^Nc_j\bold v_j,\q
c_j=\int_\O\bold v_j\cdot[\sigma(w)^{-1}\nabla\times\H_w]dx,\q j=1,\cdots, N.
$$
From this and \eqref{est3.29} we get
\eq\label{est-cj}
|c_j|\leq C_1\|\nabla\times\H^0\|_{L^2(\O)},\q j=1,\cdots, N,
\eeq
where $C_1$ depends on $\O, \sigma_1,\sigma_2$.

It is not difficult to see that $\varphi_w$ satisfies the following equation
\begin{equation}\label{eq3.30}
\left\{\aligned
&\nabla\cdot[\sigma(w)(\nabla\varphi_w+\h_w)]=0 \q & \text{in } \O,\\
&\nu\cdot[\sigma(w)(\nabla\varphi_w+\h_w)]=\nu\cdot\nabla\times\H^0\q &\text{on } \p\O.
\endaligned\right.
\end{equation}
Applying the De Giorgi-Nash estimate for elliptic equations with Neumann boundary condition (see \cite[Proposition 3.6]{Nittka2011}) to \eqref{eq3.30}, we see that there exist $\alpha=\alpha(\O,q,\sigma_1,\sigma_2)\in (0,1)$ and $C_2=C_2(\O,q,\sigma_1,\sigma_2)$ such that
\eq\label{est3.31}
\|\varphi_w\|_{C^{0,\alpha}(\overline{\O})}\leq C_2\|\nabla\times\H^0\|_{L^q(\O)}.
\eeq
\vskip0.05in

{\it Step 3}.
For $w$ and  $\H_w$ given above, we look for a solution $u_w$ of \eqref{equation-u}.
Using the decomposition \eqref{decom-3.30} we have
$$
\aligned
\sigma(w)^{-1}|\nabla\times\H_w|^2&=(\nabla\varphi_w+\h_w)\cdot\nabla\times\H_w\\
=&\nabla\cdot(\varphi_w\,\nabla\times\H_w)+\h_w\cdot\nabla\times\H_w\in H^{-1}(\O).
\endaligned
$$
So we can use the Lax-Milgram theorem to conclude that \eqref{equation-u} has a unique weak solution $u_w\in H^1(\O)$. Moreover, we have the following estimate
\eq\label{est3.34}
\aligned
\|u_w\|_{L^2(\O)}&\leq C_3\{\|\nabla\times\H^0\|_{L^q(\O)}\|\nabla\times\H^0\|_{L^2(\O)}+\|u^0\|_{H^1(\O)}\},
\endaligned
\eeq
where $C_3$ depends on $\O,q,\sigma_1,\sigma_2$.
\vskip0.05in

{\it Step 4}.
For any given function $w\in L^2(\O)$, let $u_w$ be the solution of \eqref{equation-u}, where $\H_w$ is the solution of \eqref{eq3.28} associated with $w$.
We define $\mathrm{T}(w)=u_w$. Then $\mathrm{T}$  is a map from $L^2(\O)$ to $L^2(\O)$. Having the estimates \eqref{est3.31} and \eqref{est3.34}, we can apply the Schauder's fixed point theorem to get a solution of \eqref{Pan-eq}. The details are similar to the counterpart in the proof of Theorem \ref{PZ1}, and are hence omitted.
\end{proof}

\begin{Prop}[Regularity of weak solutions]
Assume that $\O$ is a bounded $C^2$ domain in $\Bbb R^3$ and $\sigma$ satisfies \eqref{cond-rho}. Let $(u,\H)$ be an $H^1(\O)\times [H(\divg0,\O)\cap H(\curl,\O)]$ weak solution of \eqref{Pan-eq}.  Then we have the following conclusions.
\begin{itemize}
\item[(i)] There exists $p_0>2$ such that if $(u^0,\H^0)\in W^{2,p/2}(\O)\times W^{1,p}(\O,\Bbb R^3)$ with $2<p<p_0$, then $(u,\H)\in W^{2,p/2}(\O)\times W^{1,p}(\O,\Bbb R^3)$.
\item[(ii)] Assume $(u^0,\H^0)\in W^{1,q}(\O)\times W^{1,q}(\O,\Bbb R^3)$ for some $q>3$. Then there exists $\delta=\delta(\O,\sigma_1,\sigma_2)\in (0,1)$ such that for any $1<\mu<1+2\min\{\delta,1-3/q\}$, we have $u\in C^{0,(\mu-1)/2}(\overline{\O})$ with the estimate
\begin{equation*}
\|u\|_{C^{0,(\mu-1)/2}(\overline{\O})}\leq C(\O,\mu,q,\sigma_1,\sigma_2)(\|\nabla\times\H^0\|_{L^q(\O)}^2+\|u^0\|_{W^{1,q}(\O)}),
\end{equation*}
Moreover, we have
$(u,\H)\in W^{1,q}(\O)\times W^{1,q}(\O,\Bbb R^3).$
\item[(iii)] Assume  $\p\O$ is connected $($hence $\Bbb H_D(\O)=\{\0\})$. Let $\a\in(0,1)$.
 \begin{itemize}
 \item[(a)] If $\O$ is of class $C^{2,\a}$, $\sigma\in C^{0,\a}_{\loc}(\Bbb R)$, and $(u^0,\H^0)\in C^{1,\a}(\overline{\O})\times C^{1,\a}(\overline{\O},\Bbb R^3)$, then $(u, \H)\in C^{1,\a}(\overline{\O})\times C^{1,\a}(\overline{\O},\Bbb R^3)$.
 \item[(b)] If $\O$ is of class $C^{3,\a}$, $\sigma\in C^{1,\a}_{\loc}(\Bbb R)$, and $(u^0,\H^0)\in C^{2,\a}(\overline{\O})\times C^{2,\a}(\overline{\O},\Bbb R^3)$, then $(u, \H)\in C^{2,\a}(\overline{\O})\times C^{2,\a}(\overline{\O},\Bbb R^3)$.
 \end{itemize}
\end{itemize}
\end{Prop}

\begin{proof}
Conclusion (i) follows from Lemma \ref{curl0} and Meyers' estimate (\cite[Theorem 2]{GM1999}).

Now we prove (ii). Since $\nabla\times[\sigma(u)^{-1}\nabla\times\H]=\0\text{ in }\O,$
by Lemma \ref{curl0}, there exist $\varphi\in \dot{H}^1(\O)$ and $\h\in \Bbb H_N(\O)$ such that $\sigma(w)^{-1}\nabla\times\H=\nabla\varphi+\h$, where $\varphi$ solves the equation
\begin{equation*}
\left\{\aligned
&\nabla\cdot[\sigma(u)(\nabla\varphi+\h)]=0 \q & \text{in } \O,\\
&\nu\cdot[\sigma(u)(\nabla\varphi+\h)]=\nu\cdot\nabla\times\H^0\q &\text{on } \p\O.
\endaligned\right.
\end{equation*}
Then there exists $\delta=\delta(\O,\sigma_1,\sigma_2)\in (0,1)$ such that for any $1<\mu<1+2\min\{\delta,1-3/q\}$, it holds that
$$\|\varphi\|_{C^{0,(\mu-1)/2}(\overline{\O})}+\|\nabla\varphi\|_{L^{2,\mu}(\O)}\leq C(\O,\mu,q,\sigma_1,\sigma_2)\|\nabla\times H^0\|_{L^q(\O)}.$$
Since
$$
\aligned
\sigma(u)^{-1}|\nabla\times\H|^2&=\nabla\cdot[\varphi\nabla\times\H]+\h\cdot\nabla\times\H,
\endaligned
$$
we can write the equation for $u$ in the following form:
\begin{equation*}
-\Delta u=\nabla\cdot[\varphi\nabla\times\H]+\h\cdot\nabla\times\H \q\text{in } \O,\q
u=u^0\q\text{ on } \p\O.
\end{equation*}
Similarly to the proof of Theorem \ref{Thm-Lp-reg}, we obtain
\begin{equation*}
\|u\|_{C^{0,(\mu-1)/2}(\overline{\O})}\leq C(\|\nabla\times\H^0\|_{L^q(\O)}^2+\|u^0\|_{W^{1,q}(\O)}).
\end{equation*}
The rest of proof is similar to the counterpart for \eqref{PZ0}, and is hence omitted.

\end{proof}

Similarly to Theorem \ref{Thm-uniq}, we also have small boundary data uniqueness for \eqref{Pan-eq}.
\begin{Prop} Let $\O$ be a bounded domain in $\Bbb R^3$ with a $C^2$ boundary, and $\sigma$ satisfy \eqref{cond-rho} and \eqref{Lip-L}.
 Assume $(u^0,\H^0)\in W^{1,q}(\O)\times W^{1,q}(\O,\Bbb R^3)$ for some $q>3$. If $\|\nabla\times\H^0 \|_{L^q(\O)}$ is sufficiently small, then the solution of \eqref{Pan-eq} in the space $H^1(\O)\times [H(\divg0,\O)\cap H(\curl,\O)\cap(\H^0+\Bbb H_D(\O)^\perp)]$ is unique.
\end{Prop}

\vskip0.1in

\subsection*{Acknowledgements.}
The authors would like to thank the referees and the editors for valuable comments and suggestions that helped to improve the paper. This work was partially supported by the National Natural Science Foundation of China Grant Nos. 11671143 and 11431005. Zhang was also supported by Anhui Provincial Natural Science Foundation Grant No. 1908085QA28.

\end{document}